\documentclass[10pt,twocolumn]{IEEEtran}

\IEEEoverridecommandlockouts                              % This command is only
\usepackage{amsmath,graphicx,amsfonts,amssymb,amsthm,epsfig,mathrsfs,balance}

\usepackage{array}
\usepackage{subcaption}
\usepackage{tikz}
\usetikzlibrary{shapes}
\usepackage{color}
\usepackage{multirow}
\usepackage{rotating}
\usepackage{algorithm,algpseudocode,algorithmicx}
\usepackage{cite,url}
\usepackage{framed}
\usepackage{bm}
\usepackage{stmaryrd}

\newtheorem{theorem}{Theorem}

\newtheorem{definition}{Definition}
\newtheorem{lemma}{Lemma}

\newtheorem{remark}{Remark}

\renewcommand{\det}{{\mathrm{det}}}

\newcommand{\differential}{{\rm{d}}}
\newcommand{\prox}{{\mathrm{prox}}}

\DeclareMathOperator{\sech}{sech}

\title{\LARGE \bf
Hopfield Neural Network Flow: A Geometric Viewpoint
}

\author{Abhishek Halder, Kenneth F. Caluya, Bertrand Travacca, and Scott J. Moura% <-this % stops a space
\thanks{Abhishek Halder and Kenneth F. Caluya are with the Department of Applied Mathematics, University of California, Santa Cruz, CA 95064, USA,
        {\tt\small{\{ahalder,kcaluya\}@ucsc.edu}}
        
       Bertrand Travacca and Scott J. Moura are with the Department of Civil and Environmental Engineering, University of California at Berkeley, CA 94720, USA,
        {\tt\small{\{bertrand.travacca,smoura\}@berkeley.edu}}% 
}}

\begin{document}

\maketitle
\thispagestyle{empty}
\pagestyle{empty}

%%%%%%%%%%%%%%%%%%%%%%%%%%%%%%%%%%%%%%%%%%%%%%%%%%%%%%%%%%%%%%%%%%%%%%%%%%%%%%%%
\begin{abstract}
We provide gradient flow interpretations for the continuous-time continuous-state Hopfield neural network (HNN). The ordinary and stochastic differential equations associated with the HNN were introduced in the literature as analog optimizers, and were reported to exhibit good performance in numerical experiments. In this work, we point out that the deterministic HNN can be transcribed into Amari's natural gradient descent, and thereby uncover the explicit relation between the underlying Riemannian metric and the activation functions. By exploiting an equivalence between the natural gradient descent and the mirror descent, we show how the choice of activation function governs the geometry of the HNN dynamics.

For the stochastic HNN, we show that the so-called ``diffusion machine", while not a gradient flow itself, induces a gradient flow when lifted in the space of probability measures. We characterize this infinite dimensional flow as the gradient descent of certain free energy with respect to a Wasserstein metric that depends on the geodesic distance on the ground manifold. Furthermore, we demonstrate how this gradient flow interpretation can be used for fast computation via recently developed proximal algorithms.
\end{abstract}

%%%%%%%%%%%%%%%%%%%%%%%%%%%%%%%%%%%%%%%%%%%%%%%%%%%%%%%%%%%%%%%%%%%%%%%%%%%%%%%%


%%%%%%%%%%%%%%%%%%%%%%%%%%%%%%%%%%%%%%%%%%%%%%%%%%%%%%%%%%%%%%%%%%%%%%%%%%%%%%%%

\section{Introduction}
Computational models such as the Hopfield neural networks (HNN) \cite{hopfield1982neural,hopfield1984neurons} have long been motivated as analog machines to solve optimization problems. Both deterministic and stochastic differential equations for the HNN have appeared in the literature, and the numerical experiments using these models have reported appealing performance \cite{tank1986simple,kennedy1988neural,narendra1990identification,li2001hopfield,travacca2018dual}. In \cite{wong1991stochastic}, introducing the stochastic HNN as the ``diffusion machine", Wong compared the same with other stochastic optimization algorithms, and commented: ``As such it is closely related to both the Boltzmann machine and
the Langevin algorithm, but superior to both in terms of possible integrated circuit realization", and that ``A substantial speed advantage for diffusion machines is likely". Despite these promising features often backed up via numerical simulations, studies seeking basic understanding of the evolution equations remain sparse \cite{liao1998robust,joya2002hopfield}. The purpose of this paper is to revisit the differential equations for the HNN from a geometric standpoint. 

Our main contribution is to clarify the interplay between the Riemannian geometry induced by the choice of HNN activation functions, and the associated variational interpretations for the HNN flow. Specifically, in the deterministic case, we argue that the HNN flow is natural gradient descent (Section \ref{SectionNatGradDescent}) of a function over a manifold $\mathcal{M}$ with respect to (w.r.t.) a distance induced by certain metric tensor $\bm{G}$ that depends on the activation functions of the HNN. This identification leads to an equivalent mirror descent interpretation (Section \ref{SectionMirrorDescent}). We work out an example and a case study (Section \ref{SectionExample}) to highlight these ideas.
 
In the stochastic HNN, the random fluctuations resulting from the noise do not allow the sample paths of the stochastic differential equation to be interpreted as gradient flow on $\mathcal{M}$. However, we argue (Section \ref{SectionStocHNN}) that the stochastic sample path evolution engenders a (deterministic) flow of probability measures supported on $\mathcal{M}$ which does have an infinite dimensional gradient descent interpretation. In other words, the gradient flow interpretation for the stochastic HNN holds in the macroscopic or ensemble sense. This viewpoint seems novel in the context of stochastic HNN, and indeed leads to proximal algorithms enabling fast computation (Section \ref{SubsecProxAlgo}). A comparative summary of the gradient flow interpretations for the deterministic and the stochastic HNN is provided in Table \ref{TableGradFlowHNN} (after Section \ref{SubsecProxAlgo}).

\subsubsection*{Notations} The notation $\nabla$ stands for the Euclidean gradient operator. We sometimes put a subscript as in $\nabla_{\bm{q}}$ to denote the gradient w.r.t. vector $\bm{q}$, and omit the subscript when its meaning is obvious from the context. As usual, $\nabla$ acting on vector-valued function returns the Jacobian. The symbol $\nabla^{2}$ is used for the Hessian. We use the superscript $^{\top}$ to denote matrix transposition, and $\mathbb{E}_{\rho}\left[\cdot\right]$ to denote the expectation operator w.r.t. the probability density function (PDF) $\rho$, i.e., $\mathbb{E}_{\rho}\left[\cdot\right] := \int_{\mathcal{M}}(\cdot)\rho(\bm{x})\:\differential\bm{x}$. In the development  that follows, we assume that all probability measures are absolutely continuous, i.e., the corresponding PDFs exist. The notation $\mathcal{P}_{2}\left(\mathcal{M}\right)$ stands for the set of all joint PDFs supported on $\mathcal{M}$ with finite second raw moments, i.e., $\mathcal{P}_{2}\left(\mathcal{M}\right) := \{\rho : \mathcal{M} \mapsto \mathbb{R}_{\geq 0} \mid \int_{\mathcal{M}}\rho\:\differential\bm{x} = 1, \mathbb{E}_{\rho}\left[\bm{x}^{\top}\bm{x}\right] < \infty\}$. The inequality $\bm{A} \succ \bm{0}$ means that the matrix $\bm{A}$ is symmetric positive definite. The symbol $\bm{I}$ denotes an identity matrix of appropriate dimension; $\bm{1}$ denotes a column vector of ones. For $\bm{x}\in\mathcal{M}$, the symbol $\mathcal{T}_{\bm{x}}\mathcal{M}$ denotes the tangent space of $\mathcal{M}$ at $\bm{x}$; we use $\mathcal{T}\mathcal{M}$ to denote the tangent space at a generic point in $\mathcal{M}$. The set of natural numbers is denoted as $\mathbb{N}$. We use the symbols $\odot$ and $\oslash$ for denoting element-wise (i.e., Hadamard) product and division, respectively. 

%%%%%%%%%%%%%%%%%%%%%%%%%%%%%%%%%%%%%%%%%%%%%%%%%%%%%%%%%%%%%%%%%%%%%%%%%%%%%%%%

\section{HNN Flow as Natural Gradient Descent}\label{SectionNatGradDescent}
We consider the deterministic HNN dynamics given by
\begin{subequations}\label{CommuAsso}
\begin{align}
	\dot{\bm{x}}_{\rm{H}} &= -\nabla f(\bm{x}), \label{HNNode}\\
\bm{x} &= \bm{\sigma}\left(\bm{x}_{\rm{H}}\right), \label{sigmaMap}
\end{align}
\label{HNNivp}
\end{subequations}
with initial condition $\bm{x}(t=0) \in (0,1)^{n}$. Here, the function $f:\mathbb{R}^{n} \mapsto \mathbb{R}$ is continuously differentiable, and therefore admits a minimum on $[0,1]^{n}$. Without loss of generality, we assume $f$ to be non-negative on $[0,1]^{n}$. The vectors $\bm{x},\bm{x}_{\rm{H}} \in \mathbb{R}^{n}$ are referred to as the state and hidden state, respectively. The so called ``activation function" $\bm{\sigma}:\mathbb{R}^{n}\mapsto[0,1]^{n}$ is assumed to satisfy the following properties:
\begin{description}
\item[$\bullet$] homeomorphism,
\item[$\bullet$] differentiable almost everywhere,
\item[$\bullet$] has structure 
\begin{eqnarray}
	\left(\sigma_{1}({x_{\rm{H}}}_{1}), \sigma_{2}({x_{\rm{H}}}_{2}), \hdots, \sigma_{n}({x_{\rm{H}}}_{n})\right)^{\top},
	\label{componentStructure}
\end{eqnarray}
i.e., its $i$-th component depends only on the $i$-th component of the argument,
\item[$\bullet$] $\sigma_{i}(\cdot)$ for all $i=1,\hdots,n$ are strictly increasing.	
\end{description}
Examples of activation function include the logistic function, hyperbolic tangent function, among others.

To view the continuous-time dynamics (\ref{HNNivp}) from a geometric standpoint, we start by rewriting it as  
\begin{eqnarray}
\dot{\bm{x}} = -\left(\nabla_{\bm{x}_{\rm{H}}}\bm{\sigma}\right)\big\vert_{\bm{x}_{\rm{H}} = \bm{\sigma}^{-1}(\bm{x})} \nabla f(\bm{x}).
\label{dxdt}	
\end{eqnarray}
In words, the flow of $\bm{x}(t)$ is generated by a vector field that is negative of the Jacobian of $\bm{\sigma}$ evaluated at $\bm{x}_{\rm{H}} = \bm{\sigma}^{-1}(\bm{x})$, times the gradient of $f$ w.r.t. $\bm{x}$. Due to (\ref{componentStructure}), the Jacobian in (\ref{dxdt}) is a diagonal matrix. Furthermore, since each $\sigma_{i}$ is strictly increasing and differentiable almost everywhere, the diagonal terms are positive. Therefore, the right-hand-side (RHS) of (\ref{dxdt}) has the following form: negative of a positive definite matrix times the gradient of $f$. This leads to the result below.

\begin{theorem}\label{ThmNatGradFlow}
The flow $\bm{x}(t)$ governed by (\ref{dxdt}) (equivalently by (\ref{HNNivp})) is a natural gradient descent for function $f(\bm{x})$ on the manifold $\mathcal{M} = (0,1)^{n}$ with metric tensor $\bm{G} = [g_{ij}]_{i,j=1}^{n}$, given by
\begin{eqnarray}
g_{ij} = \begin{cases}
	1/\sigma^{\prime}_{i}(\sigma_{i}^{-1}(x_{i})) & \text{for}\quad i=j,\\
	0 & \text{otherwise,}
\end{cases}
\label{Mij}	
\end{eqnarray}
where $i,j=1,\hdots,n$, and $^{\prime}$ denotes derivative.
\end{theorem}
\begin{proof}
Amari's natural gradient descent \cite{amari1998natural} describes the steepest descent of $f(\bm{x})$ on a Riemannian manifold $\mathcal{M}$ with metric tensor $\bm{G} \succ \bm{0}$, and is given by the recursion
\begin{eqnarray}
\bm{x}(k+1) = \bm{x}(k) - h\left(\bm{G}(\bm{x})\right)^{-1}\nabla f(\bm{x})\big\vert_{\bm{x}=\bm{x}(k)},
\label{AmariNatGradDescent}
\end{eqnarray}	
where the discrete iteration index $k\in\mathbb{N}$, and the step-size $h>0$ is small. For $h \downarrow 0$, we get the natural gradient dynamics
\begin{eqnarray}
\dot{\bm{x}} = -\left(\bm{G}(\bm{x})\right)^{-1}\nabla f(\bm{x}).
\label{AmariNatGradODE}
\end{eqnarray}
Notice that due to the assumptions listed on the activation function $\bm{\sigma}$, its Jacobian appearing in (\ref{dxdt}) is a diagonal matrix with positive entries.	Therefore, (\ref{dxdt}) is of the form (\ref{AmariNatGradODE}) with $\bm{G}(\bm{x}) = {\rm{diag}}(1/\sigma^{\prime}_{i}(\sigma_{i}^{-1}(x_{i})))$, $i=1,\hdots,n$. In particular, $\bm{G}(\bm{x})\succ\bm{0}$ for all $\bm{x}\in(0,1)^{n}$. This completes the proof.
\end{proof} 

Theorem \ref{ThmNatGradFlow} allows us to interpret the flow generated by (\ref{dxdt}) as the steepest descent of $f$ in a Riemannian manifold $\mathcal{M}$ with length element $\differential s$ given by
\begin{eqnarray}
\differential s^{2} = \displaystyle\sum_{i=1}^{n} \displaystyle\frac{\differential x_{i}^{2}}{\sigma^{\prime}_{i}(\sigma_{i}^{-1}(x_{i}))},
\label{lineelement}
\end{eqnarray}
i.e., an orthogonal coordinate system can be associated with $\mathcal{M}$. Further insights can be obtained by interpreting $\mathcal{M}$ as a Hessian manifold \cite{shima2007geometry}. To this end, we would like to view the positive definite metric tensor $\bm{G} = {\rm{diag}}(1/\sigma^{\prime}_{i}(\sigma_{i}^{-1}(x_{i})))$ as the Hessian of a (twice differentiable) strictly convex function. Such an identification will allow us to associate a mirror descent \cite{nemirovskii1979comp lexity,beck2003mirror} in the dual manifold  \cite{raskutti2015information} corresponding to the natural gradient descent in manifold $\mathcal{M}$. %An elementary but useful observation here is that since $\frac{\partial^{2}h}{\partial x_{i}^{2}}$ must only depend on $x_{i}$, and the Hessian is diagonal, therefore $h(\bm{x})$ must be a separable sum.
Before delving into mirror descent, we next point out that (\ref{lineelement}) allows us to compute geodesics on the manifold $\mathcal{M}$, which completes the natural gradient descent interpretation. 

\subsection{Geodesics}
The (minimal) geodesic curve $\bm{\gamma}(t)$ parameterized via $t\in[0,1]$ that connects $\bm{x},\bm{y}\in(0,1)^{n}$, solves the well-known Euler-Lagrange equation (see e.g., \cite[Ch. 3.2]{docarmo1992Riemannian})
\begin{eqnarray}
\ddot{\gamma}_{k}(t) + \displaystyle\sum_{i,j=1}^{n}\Gamma_{ij}^{k}(\bm{\gamma}(t)) \dot{\gamma}_{i}(t)\dot{\gamma}_{j}(t) = 0, \: k=1,\hdots,n,
\label{GeodesicELode}	
\end{eqnarray}
subject to the boundary conditions $\bm{\gamma}(0) = \bm{x}$, $\bm{\gamma}(1) = \bm{y}$. Here $\Gamma_{ij}^{k}$ denote the Cristoffel symbols of the second kind for $i,j,k=1,\hdots,n$, and are given by
\begin{eqnarray}
\Gamma_{ij}^{k} := \frac{1}{2}\displaystyle\sum_{l=1}^{n} g^{kl} \left(\dfrac{\partial g_{jl}}{\partial x_{i}} + \dfrac{\partial g_{il}}{\partial x_{j}} - \dfrac{\partial g_{ij}}{\partial x_{l}}\right),
\label{Christoffel2ndKind}	
\end{eqnarray}
wherein $g^{kl}$ denotes the $(k,l)$-th element of the inverse metric tensor $\bm{G}^{-1}$. For $i\neq j\neq k$, and diagonal $\bm{G}$, we get
\begin{subequations}\label{diagG}
\begin{align}
&\Gamma_{ij}^{k} = 0, \label{Gammaijk}\\
&\Gamma_{ii}^{k} = -\dfrac{1}{2}g^{kk} \dfrac{\partial g_{ii}}{\partial x_{k}}, \label{Gammaiik}\\
&\Gamma_{ik}^{k} = \dfrac{\partial}{\partial x_{i}} \left(\log\sqrt{|g_{kk}|}\right), \label{Gammaikk}\\
&\Gamma_{kk}^{k} = \dfrac{\partial}{\partial x_{k}} \left(\log\sqrt{|g_{kk}|}\right).	\label{Gammakkk}	
\end{align}
\end{subequations}
For our diagonal metric tensor (\ref{Mij}), since the $i$-th diagonal element of $\bm{G}$ only depends on $x_{i}$, the terms (\ref{Gammaijk})--(\ref{Gammaikk}) are all zero, i.e., the only non-zero Cristoffel symbols in (\ref{diagG}) are (\ref{Gammakkk}). This simplifies (\ref{GeodesicELode}) as a set of $n$ \emph{decoupled} ODEs:
\begin{align}
\ddot{\gamma}_{i}(t) + \Gamma_{ii}^{i}(\bm{\gamma}(t)) \left(\dot{\gamma}_{i}(t)\right)^{2} = 0, \quad i=1,\hdots,n.
\label{DecoupledGeodesicODE}	
\end{align}
Using (\ref{Mij}) and (\ref{Gammakkk}), for $i=1,\hdots,n$, we can rewrite (\ref{DecoupledGeodesicODE}) as
\begin{align}
\ddot{\gamma}_{i}(t) + \dfrac{\sigma_{i}^{\prime\prime}\left(\sigma_{i}^{-1}\left(\gamma_{i}\right)\right)\sigma_{i}^{\prime}\left(\gamma_{i}\right)}{2\:\sigma_{i}^{\prime}\left(\sigma_{i}^{-1}\left(\gamma_{i}\right)\right) \left(\sigma_{i}\left(\gamma_{i}\right)\right)^{2}} \left(\dot{\gamma}_{i}(t)\right)^{2} = 0,
\label{DecoupledExplicitODE}	
\end{align}
which solved together with the endpoint conditions $\bm{\gamma}(0) = \bm{x}$, $\bm{\gamma}(1) = \bm{y}$, yields the geodesic curve $\bm{\gamma}(t)$. In Section \ref{SectionExample}, we will see an explicitly solvable example for the system of nonlinear ODEs (\ref{DecoupledExplicitODE}). 

Once the geodesic $\bm{\gamma}(t)$ is obtained from (\ref{DecoupledExplicitODE}), the geodesic distance $d_{\bm{G}}(\bm{x},\bm{y})$ induced by the metric tensor $\bm{G}$, can be computed as
\begin{subequations}\label{GeodesicDist}	
\begin{align}
d_{\bm{G}}(\bm{x},\bm{y}) &:= \displaystyle\int_{0}^{1} \left(\displaystyle\sum_{i,j=1}^{n}g_{ij}(\bm{\gamma}(t))\dot{\gamma}_{i}(t)\dot{\gamma}_{j}(t)\right)^{\!\!1/2}{\rm{d}}t \label{GeodesicDistGeneral}\\
&= \displaystyle\int_{0}^{1} \left(\displaystyle\sum_{i=1}^{n}\dfrac{\left(\dot{\gamma}_{i}(t)\right)^{2}}{\sigma^{\prime}_{i}(\sigma_{i}^{-1}(\gamma_{i}(t)))}\right)^{\!\!1/2}{\rm{d}}t.\label{GeodesicDistSpecific}
\end{align}
\end{subequations}
This allows us to formally conclude that the flow generated by (\ref{dxdt}), or equivalently by (\ref{AmariNatGradODE}), can be seen as the gradient descent of $f$ on the manifold $\mathcal{M}$, measured w.r.t. the distance $d_{\bm{G}}$ given by (\ref{GeodesicDistSpecific}).

\begin{remark}\label{RiemannianGradient}
In view of the notation $d_{\bm{G}}$, we can define the Riemannian gradient $\nabla^{d_{\bm{G}}}\left(\cdot\right) := (\bm{G}(\bm{x}))^{-1}\nabla\left(\cdot\right)$, and rewrite the dynamics (\ref{AmariNatGradODE}) as $\dot{\bm{x}} = -\nabla^{d_{\bm{G}}}f(\bm{x})$.  	
\end{remark}

%%%%%%%%%%%%%%%%%%%%%%%%%%%%%%%%%%%%%%%%%%%

\section{HNN Flow as Mirror Descent}\label{SectionMirrorDescent}
%Section \ref{SectionNatGradDescent} revealed that (\ref{dxdt}) can be viewed as a natural gradient descent, which helps to interpret the flow generated by (\ref{dxdt}) as a steepest descent of $f$ measured (\ref{lineelement}). 

We now provide an alternative way to interpret the flow generated by (\ref{dxdt}) as mirror descent on a Hessian manifold $\mathcal{N}$ that is dual to $\mathcal{M}$. For $\bm{x}\in\mathcal{M}$ and $\bm{z}\in\mathcal{N}$, consider the mapping $\mathcal{N} \mapsto \mathcal{M}$ given by the map $\bm{x} = \nabla_{\bm{z}} \psi(\bm{z})$ for some strictly convex and differentiable map $\psi(\cdot)$, i.e., 
\[(\nabla\psi)^{-1} : \mathcal{M} \mapsto \mathcal{N}.\] 
The map $\psi(\cdot)$ is referred to as the ``mirror map", formally defined below. In the following, the notation ${\rm{dom}}(\psi)$ stands for the domain of $\psi(\cdot)$.

\begin{definition} \label{DefMirrorMap}(\textbf{Mirror map})
A differentiable, strictly convex function $\psi : {\rm{dom}}(\psi) \mapsto \mathbb{R}$, on an open convex set ${\rm{dom}}(\psi) \subseteq \mathbb{R}^{n}$, satisfying $\parallel \nabla \psi\parallel_{2} \rightarrow\infty$ as the argument of $\psi$ approaches the boundary of the closure of ${\rm{dom}}(\psi)$, is called a mirror map.
\end{definition}
Given a mirror map, one can introduce a Bregman divergence \cite{bregman1967} associated with it as follows.
\begin{definition} \label{DefBregDiv}(\textbf{Bregman divergence associated with a mirror map})
For a mirror map $\psi$, the associated Bregman divergence $D_{\psi} : {\rm{dom}}(\psi) \times {\rm{dom}}(\psi) \mapsto \mathbb{R}_{\geq 0}$, is given by
\begin{align}
D_{\psi}\left(\bm{z}, \widetilde{\bm{z}}\right) := \psi(\bm{z}) - \psi(\widetilde{\bm{z}}) - (\bm{z} - \widetilde{\bm{z}})^{\top}\nabla\psi( \widetilde{\bm{z}}).
\label{BregDivFormula}
\end{align}
\end{definition}
Geometrically, $D_{\psi}\left(\bm{z}, \widetilde{\bm{z}}\right)$ is the error at $\bm{z}$ due to first order Taylor approximation of $\psi$ about $\widetilde{\bm{z}}$. It is easy to verify that $D_{\psi}\left(\bm{z}, \widetilde{\bm{z}}\right) = 0$ iff $\bm{z}=\widetilde{\bm{z}}$. However, the Bregman divergence is not symmetric in general, i.e., $D_{\psi}\left(\bm{z}, \widetilde{\bm{z}}\right) \neq D_{\psi}\left(\widetilde{\bm{z}},\bm{z}\right)$.

In the sequel, we assume that the mirror map $\psi$ is twice differentiable. Then, $\psi$ being strictly convex, its Hessian $\nabla^{2}\psi$ is positive definite. This allows one to think of $\mathcal{N}\equiv{\rm{dom}}(\psi)$ as a Hessian manifold with metric tensor $\bm{H}\equiv\nabla^{2}\psi$. The mirror descent for function $f$ on a convex set $\mathcal{Z} \subset \mathcal{N}$, is a recursion of the form
\begin{subequations}\label{MirrorDescent}
\begin{align}
\nabla\psi(\bm{y}(k+1)) &= \nabla\psi\left(\bm{z}(k)\right) - h\nabla f\left(\bm{z}(k)\right), \label{gradupdate}\\
\bm{z}(k+1) &= {\rm{proj}}^{D_{\psi}}_{\mathcal{Z}}\left(\bm{y}(k+1)\right), \label{varupdate}
\end{align}
\end{subequations}
where  $h>0$ is the step-size, and the Bregman projection
\begin{align}
{\rm{proj}}^{D_{\psi}}_{\mathcal{Z}}\left(\bm{\eta}\right) := \underset{\bm{\xi}\in\mathcal{Z}}{\arg\min}\, D_{\psi}\left(\bm{\xi}, \bm{\eta}\right).
\label{BregProj}
\end{align}
For background on mirror descent, we refer the readers to \cite{nemirovskii1979comp lexity,beck2003mirror}. For a recent reference, see \cite[Ch. 4]{bubeck2015}.

Raskutti and Mukherjee \cite{raskutti2015information} pointed out that if $\psi$ is twice differentiable, then the map $\mathcal{N}\mapsto\mathcal{M}$ given by $\bm{x} = \nabla_{\bm{z}}\psi(\bm{z})$, establishes a one-to-one correspondence between the natural gradient descent (\ref{AmariNatGradDescent}) and the mirror descent (\ref{MirrorDescent}). Specifically, the mirror descent (\ref{MirrorDescent}) on $\mathcal{Z} \subset \mathcal{N}$ associated with the mirror map $\psi$, is equivalent to the natural gradient descent (\ref{AmariNatGradDescent}) on $\mathcal{M}$ with metric tensor $\nabla^{2}\psi^{*}$, where 
\begin{align}
\psi^{*}\left(\bm{x}\right) := \underset{\bm{z}}{\sup}\left(\bm{x}^{\top}\bm{z} \:-\: \psi(\bm{z})\right),\label{LegFen}
\end{align}
is the Legendre-Fenchel conjugate \cite[Section 12]{rockafeller1970} of $\psi(\bm{z})$. In our context, the Hessian $\nabla^{2}\psi^{*}\equiv\bm{G}$ is given by (\ref{Mij}). To proceed further, the following Lemma will be useful.
\begin{lemma}\label{LemmaDuality} (\textbf{Duality})
Consider a mirror map $\psi$ and its associated Bregman divergence $D_{\psi}$ as in Definitions \ref{DefMirrorMap} and \ref{DefBregDiv}. It holds that 
\begin{subequations}\label{Duality} 
\begin{flalign}
&\nabla\psi^{*} = \left(\nabla\psi\right)^{-1}, \label{StarInv}\\
&D_{\psi}\left(\bm{z},\widetilde{\bm{z}}\right) = D_{\psi^{*}}\left(\nabla\psi(\widetilde{\bm{z}}),\nabla\psi(\bm{z})\right). \label{BregStar}
\end{flalign}
\end{subequations}
\end{lemma}
\begin{proof}
See for example, \cite[Section 2.2]{nielsen2007bregman}.
\end{proof}

Combining (\ref{Mij}) with (\ref{StarInv}), we get 
\begin{align}
z_{i} &= \left((\nabla\psi)^{-1}(\bm{x})\right)_{i} \nonumber\\
&= \frac{\partial\psi^{*}}{\partial x_{i}} = \underbrace{\int\frac{{\mathrm{d}}x_{i}}{\sigma^{\prime}_{i}(\sigma_{i}^{-1}(x_{i}))}}_{=: \phi(x_{i})} \:+\: \underbrace{\vphantom{\int\frac{{\mathrm{d}}x_{i}}{\sigma^{\prime}_{i}(\sigma_{i}^{-1}(x_{i}))}} k_{i}\left(\bm{x}_{-i}\right)}_{\substack{\text{arbitrary function that}\\ \text{does not depend on $x_{i}$}}},
\label{zasfuncofx}
\end{align} 
for $i=1,\hdots,n$, where the notation $k_{i}\left(\bm{x}_{-i}\right)$ means that the function $k_{i}(\cdot)$ does not depend on the $i$-th component of the vector $\bm{x}$. Therefore, the map $\mathcal{M} \mapsto \mathcal{N}$ associated with the HNN flow is of the form (\ref{zasfuncofx}). 

To derive the mirror descent (\ref{MirrorDescent}) associated with the HNN flow, it remains to determine $D_{\psi}$. Thanks to (\ref{BregStar}), we can do so without explicitly computing $\psi$. Specifically, integrating (\ref{zasfuncofx}) yields $\psi^{*}$, i.e.,
\begin{align}
\psi^{*}(\bm{x}) = \int\phi(x_{i}){\mathrm{d}}x_{i} + \kappa\prod_{i=1}^{n}x_{i} + c,
\label{psistarGeneralForm}
\end{align} 
where $\phi(\cdot)$ is defined in (\ref{zasfuncofx}), and $\kappa, c$ are constants. The constants $\kappa, c$ define a parametrized family of maps $\psi^{*}$. Combining (\ref{BregDivFormula}) and (\ref{psistarGeneralForm}) yields $D_{\psi^{*}}$, and hence $D_{\psi}$ (due to (\ref{BregStar})). In Section \ref{SectionExample}, we illustrate these ideas on a concrete example.
 
 \begin{remark}
 We note that one may rewrite (\ref{HNNivp}) in the form of a mirror descent ODE \cite[Section 2.1]{krichene2015accelerated}, given by
 \begin{align}
 \dot{\bm{x}}_{\rm{H}} = -\nabla f(\bm{x}), \quad \bm{x} = \nabla\psi^{*}\left(\bm{x}_{\rm{H}}\right),
 \label{MirrorDescentODE}	
 \end{align}
for some to-be-determined mirror map $\psi$, thereby interpreting (\ref{HNNivp}) as mirror descent in $\mathcal{M}$ itself, with $\bm{x}_{\rm{H}}\in\mathbb{R}^{n}$ as the dual variable, and $\bm{x}\in[0,1]^{n}$ as the primal variable. This interpretation, however, is contingent on the assumption that the monotone vector function $\bm{\sigma}$ is expressible as the gradient of a convex function (here, $\psi^{*}$), for which the necessary and sufficient condition is that the map $\bm{\sigma}$ be maximally cyclically monotone \cite[Section 2, Corollary 1 and 2]{rockafellar1966characterization}. This indeed holds
under the structural assumption (\ref{componentStructure}). 
\end{remark}

%%%%%%%%%%%%%%%%%%%%%%%%%%%%%%%%%%%%%%%%%%%

\begin{figure}[t]
\centering
\includegraphics[width=\linewidth]{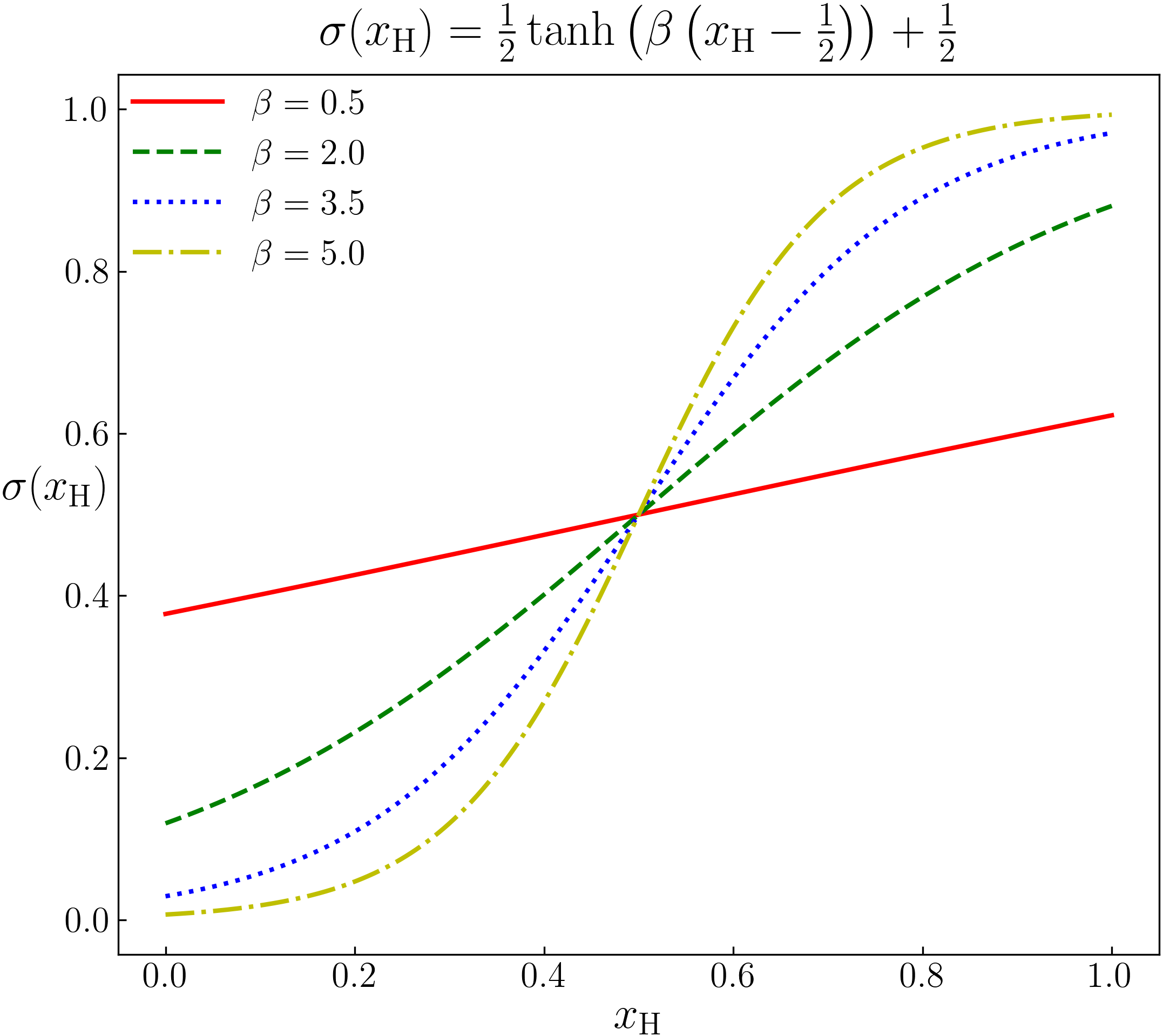}
\vspace*{-0.1in}
\caption{\small{For $x_{\rm{H}}\in[0,1]$, the activation function $\sigma(x_{\rm{H}})$ associated with the soft-projection operator (\ref{softproj}).}}
\vspace*{-0.1in}
\label{SoftProjActFunc}
\end{figure}

\section{Examples}\label{SectionExample}
In this Section, we first give a simple example (Section \ref{SubsecIllustrativeExample}) to elucidate the ideas presented in Sections \ref{SectionNatGradDescent} and \ref{SectionMirrorDescent}. Then, we provide a case study (Section \ref{SubsecCaseStudy}) to exemplify their application in the context of an engineering problem.

\subsection{An Illustrative Example}\label{SubsecIllustrativeExample}

Consider the activation function $\bm{\sigma} : \mathbb{R}^{n} \mapsto [0,1]^{n}$ given by the soft-projection operator
\begin{eqnarray}
\sigma_{i}(x_{{\rm{H}}_{i}}) := \frac{1}{2}\:\tanh\left(\beta_{i}\left(x_{{\rm{H}}_{i}} - \frac{1}{2}\right)\right) + \frac{1}{2}, \quad \beta_{i} > 0;
\label{softproj}
\end{eqnarray}
see Fig. \ref{SoftProjActFunc}. We next illustrate the natural gradient descent and the mirror descent interpretations of (\ref{dxdt}) with activation function (\ref{softproj}).

\subsubsection{Natural Gradient Descent Interpretation}
Direct calculations yield 
\begin{eqnarray}
\sigma^{\prime}_{i}(\sigma_{i}^{-1}(x_{i})) = \frac{1}{2}\beta_{i}\sech^{2} \left(\tanh^{-1}(2x_{i}-1)\right),
\label{JacDiagExample}
\end{eqnarray}
for all $i=1,\hdots,n$. Notice that $x_{i} \in [0,1]$ implies $(2x_{i}-1)\in[-1,1]$, and hence the RHS of (\ref{JacDiagExample}) is positive. We have
\begin{align}
&(\nabla^{2}\psi^{*})_{ii}=g_{ii} = 1/\sigma^{\prime}_{i}(\sigma_{i}^{-1}(x_{i})) \nonumber\\
&= \frac{2}{\beta_{i}}\cosh^{2} \left(\tanh^{-1}(2x_{i}-1)\right) = \frac{2}{\beta_{i}} \frac{1}{1 - (1-2x_{i})^{2}},
\label{HessianExample}
\end{align}
which allows us to deduce the following: the HNN flow with activation function (\ref{softproj}) is the natural gradient descent (\ref{AmariNatGradDescent}) with diagonal metric tensor $\bm{G}$ given by (\ref{HessianExample}), i.e., $g_{ii}=1/\left(2\beta_{i}x_{i}(1-x_{i})\right)$, $i=1,\hdots,n$.

%%%%%%%%%%%%%%%%%%%%%%%%%%%%%%%%%%%%%%%%%%%

\begin{figure}[t]
\centering
\includegraphics[width=\linewidth]{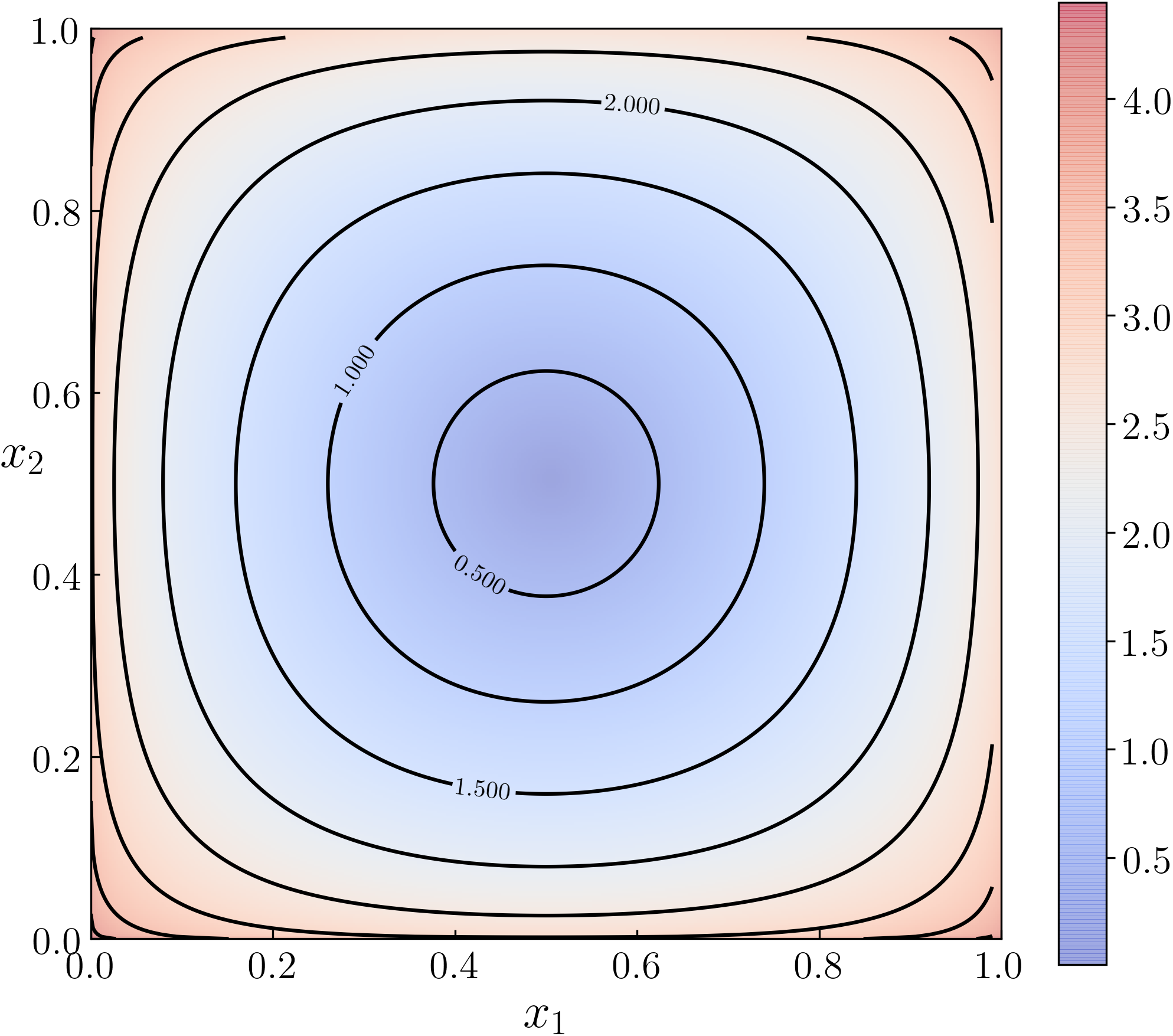}
\vspace*{-0.1in}
\caption{\small{For $\bm{x}\in[0,1]^{2}$, shown above are the contour plots for the geodesic balls of $d_{\bm{G}}$ given by (\ref{GeodesicDistSoftProj}) with $\beta_{i}=1/4$ for $i=1,2$, centered at $(0.5,0.5)$.}}
\vspace*{-0.24in}
\label{FigGeodesicBall}
\end{figure}

%%%%%%%%%%%%%%%%%%%%%%%%%%%%%%%%%%%%%%%%%%

In this case, (\ref{DecoupledExplicitODE}) reduces to the nonlinear ODE
\begin{align}
\ddot{\gamma}_{i} + \dfrac{2\gamma_{i} - 1}{2\gamma_{i}\left(1-\gamma_{i}\right)}\left(\dot{\gamma}_{i}\right)^{2} = 0,	
\label{GeodesicODESoftProj}	
\end{align} 
which solved with the boundary conditions $\gamma_{i}(0) = x_{i}$, $\gamma_{i}(1)=y_{i}$, $i=1,\hdots,n$, determines the geodesic curve $\bm{\gamma}(t)$ connecting $\bm{x},\bm{y}\in(0,1)^{n}$. Introducing the change-of-variable $v_{i} := \dot{\gamma}_{i}$, noting that $\ddot{\gamma}_{i} = v_{i}\dfrac{{\rm{d}}v_{i}}{{\rm{d}}\gamma_{i}}$, and enforcing $v_{i} \not\equiv 0$, we can transcribe (\ref{GeodesicODESoftProj}) into the separable ODE
\begin{align}
\dfrac{{\rm{d}}v_{i}}{{\rm{d}}\gamma_{i}} + \dfrac{2\gamma_{i} - 1}{2\gamma_{i}\left(1-\gamma_{i}\right)}v_{i} = 0,
\label{separableODE}	
\end{align}
which gives 
\begin{align}
v_{i} \equiv \left(\dot{\gamma}_{i}\right)^{2} = a_{i}\gamma_{i}\left(1-\gamma_{i}\right),
\label{gammadotSquared}	
\end{align}
and consequently,  
\begin{align}
\gamma_{i}(t) = \sin^{2}\left(\frac{1}{2}\left(\sqrt{a_{i}}\:t + b_{i}\right)\right), \quad k=1,\hdots,n,
\label{gammaWithConstantOfIntegration}	
\end{align}
where $\{\left(a_{i},b_{i}\right)\}_{i=1}^{n}$ are constants of integration. Using the boundary conditions $\gamma_{i}(0) = x_{i}$, $\gamma_{i}(1)=y_{i}$ in (\ref{gammaWithConstantOfIntegration}) yields the geodesic curve $\bm{\gamma}(t)$ as
\begin{align}
\gamma_{i}(t) = \sin^{2}\left(\left(1-t\right)\arcsin\sqrt{x_{i}} + t\arcsin\sqrt{y_{i}}\right), 
\label{GeodesicSoftProj}	
\end{align}
where $i=1,\hdots,n$, and $t\in[0,1]$. From (\ref{GeodesicSoftProj}), it is easy to verify that $\bm{\gamma}(t)$ is component-wise in $[0,1]$, and satisfies the boundary conditions $\bm{\gamma}(0) = \bm{x}$, $\bm{\gamma}(1) = \bm{y}$. From (\ref{GeodesicDistSpecific}) and (\ref{GeodesicSoftProj}), the geodesic distance associated with the activation function (\ref{softproj}) is
\begin{align}
d_{\bm{G}}\left(\bm{x},\bm{y}\right) = \parallel  \left(\arcsin\sqrt{\bm{x}} - \arcsin\sqrt{\bm{y}}\right)\oslash \bm{\beta} \parallel_{2},\label{GeodesicDistSoftProj}	
\end{align}
where all vector operands such as square-root, arcsin, division (denoted by $\oslash$), are element-wise. The contour plots in Fig. \ref{FigGeodesicBall} show the geodesic balls of $d_{\bm{G}}$ given by (\ref{GeodesicDistSoftProj}), centered at $(0.5,0.5)$. We summarize: the HNN flow generated by (\ref{dxdt}) with activation function (\ref{softproj}) is natural gradient descent of $f$ measured w.r.t. the geodesic distance (\ref{GeodesicDistSoftProj}).

\subsubsection{Mirror Descent Interpretation}

To derive a mirror descent interpretation, integrating (\ref{HessianExample}) we obtain
\begin{align}
z_{i}=\frac{\partial \psi^{*}}{\partial x_{i}} =  \frac{1}{2\beta_{i}} \left(\log x_{i} - \log(1-x_{i})\right) + k_{i}\left(\bm{x}_{-i}\right),
\label{GradFuncExample}
\end{align}
where $i=1,\hdots,n$, and as before, $k_{i}\left(\bm{x}_{-i}\right)$ means that the function $k_{i}(\cdot)$ does not depend on the $i$-th component of the vector $\bm{x}$. Equation (\ref{GradFuncExample}) implies that
\begin{align}
\psi^{*}(\bm{x}) = \displaystyle\sum_{i=1}^{n}\frac{1}{2\beta_{i}}\left(x_{i}\log x_{i} + (1-x_{i})\log(1-x_{i})\right) \nonumber\\
+ \kappa\prod_{i=1}^{n}x_{i} + c,
\label{FuncExample}
\end{align}
where $\kappa, c$ are constants. Therefore, (\ref{GradFuncExample})  can be re-written as
\begin{align}
z_{i}=\frac{\partial \psi^{*}}{\partial x_{i}} =  \frac{1}{2\beta_{i}} \left(\log x_{i} - \log(1-x_{i})\right) + \kappa\prod_{\stackrel{j=1}{j\neq i}}^{n}x_{j}.
\label{GradFuncExampleFinal}
\end{align}
Comparing the RHS of (\ref{GradFuncExampleFinal}) with (\ref{zasfuncofx}), for this specific example, we have that
\begin{align}
\phi(x_{i}) = \frac{1}{2\beta_{i}}\log\left(\frac{x_{i}}{1-x_{i}}\right), \quad k_{i}(\bm{x}_{-i}) = \kappa\prod_{\stackrel{j=1}{j\neq i}}^{n}x_{j}.
\label{phiandk}
\end{align}

%\subsubsection{The case $\kappa=c=0$}

Let us consider a special case of (\ref{FuncExample}), namely $\kappa=c=0$, that is particularly insightful. In this case, 
\[\psi^{*}(\bm{x}) = \sum_{i=1}^{n}\frac{1}{2\beta_{i}}(x_{i}\log x_{i} + (1-x_{i})\log(1-x_{i})),\] 
i.e., a separable (weighted) sum of the bit entropy (see \cite[Table 1]{nielsen2007bregman} and \cite[Table 1]{banerjee2005clustering}). The associated Bregman divergence
\begin{align}
D_{\psi^{*}}(\bm{\xi},\bm{\eta}) \!= \!\!\sum_{i=1}^{n}\frac{1}{2\beta_{i}}\bigg\{\xi_{i}\log\left(\frac{\xi_{i}}{\eta_{i}}\right) \!+\! \left(1-\xi_{i}\right)\log\left(\frac{1-\xi_{i}}{1-\eta_{i}}\right)\bigg\}
\label{LogisticLossBreg}
\end{align}
can be interpreted as the weighted sum of logistic loss functions associated with the Bernoulli random variables. Furthermore, $z_{i} = \log(x_{i}/(1-x_{i}))/2\beta_{i}$ yields
\begin{align}
x_{i} = \frac{\exp(2\beta_{i}z_{i})}{1 + \exp(2\beta_{i}z_{i})}.
\label{xasfuncofz}
\end{align}
Consequently, using (\ref{BregStar}) and (\ref{xasfuncofz}), we get 
\begin{align}
&D_{\psi}\left(\bm{z},\widetilde{\bm{z}}\right) = D_{\psi^{*}}\left(\widetilde{\bm{x}},\bm{x}\right) \nonumber\\
&= \sum_{i=1}^{n}\frac{1}{2\beta_{i}}\log\left(\frac{1 + \exp(2\beta_{i}z_{i})}{1 + \exp(2\beta_{i}\widetilde{z}_{i})}\right) - \left(z_{i} - \widetilde{z}_{i}\right)\frac{\exp(2\beta_{i}\widetilde{z}_{i})}{1 + \exp(42\beta_{i}\widetilde{z}_{i})},
\label{BregExampleSplCase}
\end{align}
which is the dual Logistic loss \cite[Table 1]{nielsen2007bregman}, and the corresponding mirror map can be identified as
\begin{align}
\psi(\bm{z}) = \sum_{i=1}^{n}\log\left(1 + \exp(2\beta_{i}z_{i})\right).
\label{dualbitentropy}
\end{align}
Thus, an alternative interpretation of the HNN dynamics with activation function (\ref{softproj}) is as follows: it can be seen as mirror descent of the form (\ref{MirrorDescent}) on $\mathcal{Z}\equiv\mathbb{R}^{n}$ in variables $z_{i}$ given by (\ref{GradFuncExampleFinal}); a candidate mirror map is given by (\ref{dualbitentropy}) with associated Bregman divergence (\ref{BregExampleSplCase}).

%%%%%%%%%%%%%%%%%%%%%%%%%%%%%%%%%%%%%%%%%%%

\subsection{Case Study: Economic Load Dispatch Problem}\label{SubsecCaseStudy}
We now apply the HNN gradient descent to solve an economic load dispatch problem in power system -- a context in which HNN methods have been used before \cite{travacca2018dual,park1993economic,king1995optimal,lee1998adaptive}. Our objective, however, is to illustrate the geometric ideas presented herein. 

Specifically, we consider $n_{\text{G}}$ generators which at time $t=0$ are either ON or OFF, denoted via state vector $\bm{x}_{0}\in\{0,1\}^{n_{\text{G}}}$. For $i=1,\hdots,n_{\text{G}}$, if the $i$\textsuperscript{th} generator is ON (resp. OFF) at $t=0$, then the $i$\textsuperscript{th} component of $\bm{x}_{0}$ equals 1 (resp. 0). Let $\bm{y}_{0}\in[0,1]^{n_{\text{G}}}$ be the vector of the generators' power output magnitudes at $t=0$. A dispatcher simultaneously decides for the next time step $t=1$, which of the generators should be ON or OFF ($\bm{x}\in\{0,1\}^{n_{\text{G}}}$) and the corresponding power output ($\bm{y}\in[0,1]^{n_{\text{G}}}$) in order to satisfy a known power demand $\pi_{d}>0$ while minimizing a convex cost of the form 
\begin{align}
J(\bm{x},\bm{y}) := \bm{p}^{\top}\bm{y} + \frac{1}{2}c_{1}\|\bm{x} - \bm{x}_{0}\|_{2}^{2} + \frac{1}{2}c_{2}\|\bm{y} - \bm{y}_{0}\|_{2}^{2}.
\label{CaseStudyCostFn}	
\end{align}
Here, $\bm{p}$ denotes the (known) marginal production cost of the generators; the second summand in (\ref{CaseStudyCostFn}) models the symmetric cost for switching the generator states from ON to OFF or from OFF to ON; the last summand in (\ref{CaseStudyCostFn}) models the generator ramp-up or ramp-down cost. The weights $c_{1},c_{2}>0$ are known to the dispatcher. The optimization problem the dispatcher needs to solve (over single time step) reads
\begin{subequations}
\begin{align}
&\underset{\substack{\bm{x}\in\{0,1\}^{n_{\text{G}}}\\ \bm{y}\in[0,1]^{n_{\text{G}}}}}{\text{minimize}}\quad J(\bm{x},\bm{y})\label{EDPobjective}\\
&\text{subject to}\quad \bm{x}^{\top}\bm{y} = \pi_{d},\label{EDPconstraint1}\\
&\qquad\qquad\quad\bm{1}^{\top}\bm{y} = \pi_{d}.\label{EDPconstraint2}
\end{align}
\label{CaseStudyOptProblem}		
\end{subequations}
The constraint (\ref{EDPconstraint2}) enforces that the supply equals the demand; the constraint (\ref{EDPconstraint1}) ensures that only those generators which are ON, can produce electricity. 

To solve the Mixed Integer Quadratically Constrained Quadratic Programming (MIQCQP) problem (\ref{CaseStudyOptProblem}), we construct the augmented Lagrangian 
\begin{align}
&\mathscr{L}_{r}\left(\bm{x},\bm{y},\lambda_{1},\lambda_{2}\right) := J(\bm{x},\bm{y}) + \lambda_{1}\left(\bm{x}^{\top}\bm{y} - \pi_{d}\right) \nonumber\\
&+ \lambda_{2}\left(\bm{1}^{\top}\bm{y} - \pi_{d}\right) + \frac{r}{2}\bigg\{\left(\bm{x}^{\top}\bm{y} - \pi_{d}\right)^{2} + \left(\bm{1}^{\top}\bm{y} - \pi_{d}\right)^{2}\bigg\},
\label{AugmentedLagrangian}	
\end{align}
where the parameter $r>0$, and the dual variables $\lambda_{1},\lambda_{2}\in\mathbb{R}$. Then, the augmented Lagrange dual function
\begin{align}
g_{r}\left(\lambda_{1},\lambda_{2}\right) = 
\underset{\substack{\bm{x}\in\{0,1\}^{n_{\text{G}}}\\ \bm{y}\in[0,1]^{n_{\text{G}}}}}{\text{minimize}}\quad \mathscr{L}_{r}\left(\bm{x},\bm{y},\lambda_{1},\lambda_{2}\right).
\label{AugmentedLagrangeDualFn}
\end{align}
For some initial guess of the pair $(\lambda_{1},\lambda_{2})$, we apply the HNN natural gradient descent (\ref{AmariNatGradDescent}) to perform the minimization (\ref{AugmentedLagrangeDualFn}), and then do the dual ascent updates 
\begin{subequations}
	\begin{align}
\lambda_{1} &\mapsfrom \lambda_{1} + h_{\text{dual}} \left(\bm{x}^{\top}\bm{y} - \pi_{d}\right),\\
\lambda_{2} &\mapsfrom \lambda_{2} + h_{\text{dual}}	\left(\bm{1}^{\top}\bm{y} - \pi_{d}\right), 
\end{align}
\end{subequations}
where $h_{\text{dual}}>0$ is the step size for the dual updates. Specifically, for the iteration index $k=0,1,2,...$, we perform the recursions
\begin{subequations}
\begin{align}
\left(\bm{x}(k),\bm{y}(k)\right) &= \underset{\substack{\bm{x}\in\{0,1\}^{n_{\text{G}}}\\ \bm{y}\in[0,1]^{n_{\text{G}}}}}{\arg\min}\quad \mathscr{L}_{r}\left(\bm{x},\bm{y},\lambda_{1}(k),\lambda_{2}(k)\right),\label{xyupdate}\\
\lambda_{1}(k+1) &= \lambda_{1}(k) + h_{\text{dual}} \left(\bm{x}(k)^{\top}\bm{y}(k) - \pi_{d}\right),\label{lambda1update}\\
\lambda_{2}(k+1) &= \lambda_{2}(k) + h_{\text{dual}}	\left(\bm{1}^{\top}\bm{y}(k) - \pi_{d}\right),\label{lambda2update} 	
\end{align}
\label{DualAscentRecursion}	
\end{subequations}
where (\ref{xyupdate}) is solved via Hopfield sub-iterations with step size $h_{\text{Hopfield}}>0$. In \cite{travacca2018dual}, this algorithm was referred to as the \emph{dual Hopfield method}, and was shown to have superior numerical performance compared to other standard approaches such as semidefinite relaxation.

%%%%%%%%%%%%%%%%%%%%%%%%%%%%%%%%%%%%%%%%%%%

\begin{figure}[t]
\centering
\includegraphics[width=\linewidth]{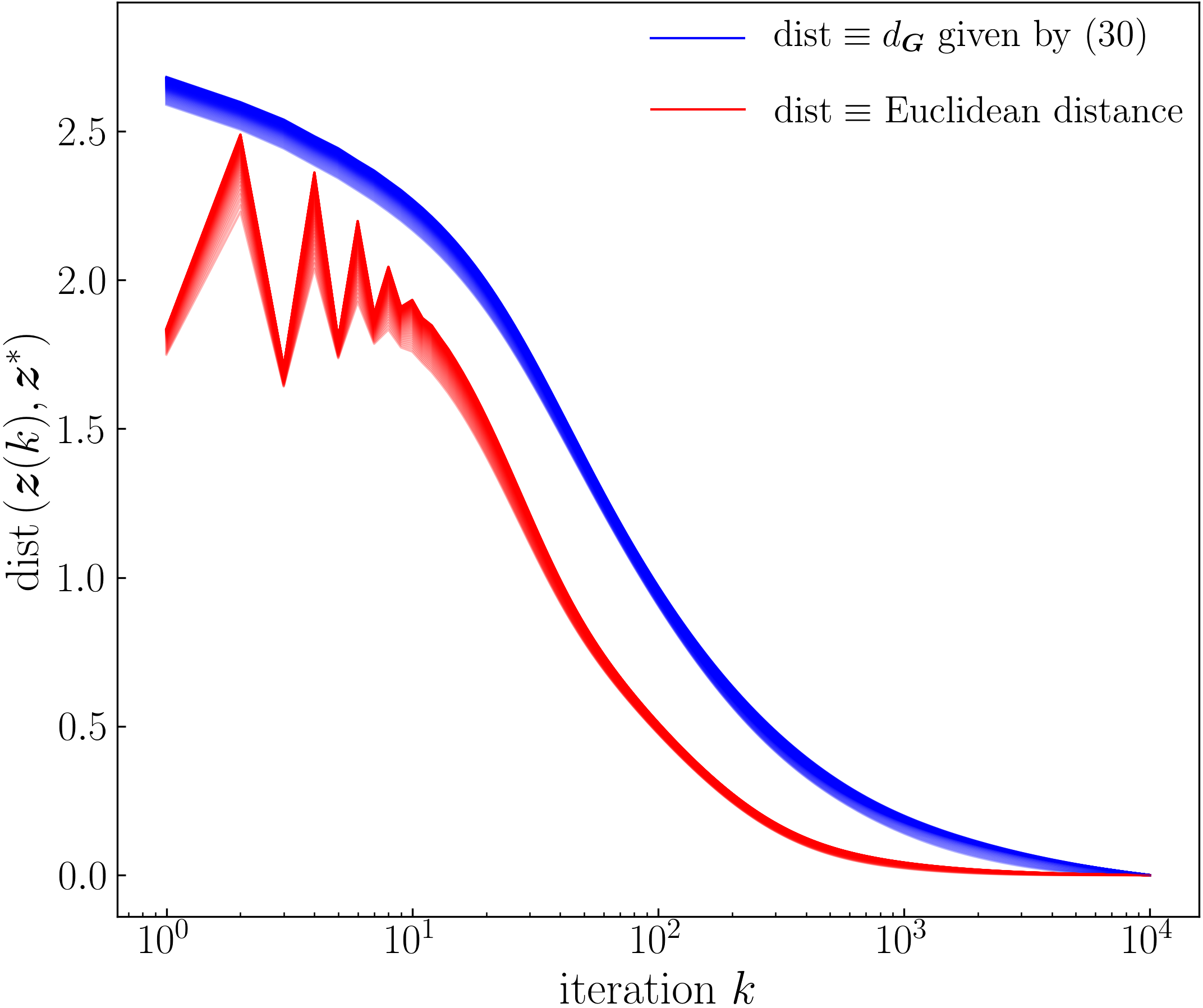}
\vspace*{-0.1in}
\caption{\small{For $\bm{z}:=\left(\bm{x};\bm{y}\right) \in \mathbb{R}^{2n_{\text{G}}}$, and $k\in\mathbb{N}$, the distance to minimum ${\rm{dist}}(\bm{z}(k),\bm{z}^{*})$ obtained from the dual Hopfield recursion (\ref{DualAscentRecursion}) applied to the economic load dispatch problem in Section \ref{SubsecCaseStudy}. Results are shown for 100 Monte Carlo simulations and two types of distances: the geodesic distance $d_{\bm{G}}$ given by (\ref{GeodesicDistSoftProj}) (depicted in \emph{blue}), and the Euclidean distance (depicted in \emph{red}).}}
\vspace*{-0.15in}
\label{CaseStdyConvergence}
\end{figure}

%%%%%%%%%%%%%%%%%%%%%%%%%%%%%%%%%%%%%%%%%%

For numerical simulation, we implement (\ref{DualAscentRecursion}) with $n_{\text{G}} = 40$ generators, uniform random vector $\bm{p}\in [0,1]^{n_{\text{G}}}$, the pair $(c_{1},c_{2})$ uniform random in $(0,1]^{2}$, $r=1$, $h_{\text{Hopfield}} = 10^{-2}$, and $h_{\text{dual}} = 10^{-1}$. The vectors $\bm{x}_{0}$ and $\bm{y}_{0}$ in (\ref{CaseStudyCostFn}) are chosen as follows: $\bm{x}_{0}$ is chosen to be uniform random in $[0,1]^{n_{\text{G}}}$ and then element-wise rounded to the nearest integers; $\bm{y}_{0}$ is chosen uniform random in $[0,1]^{n_{\text{G}}}$ and then element-wise multiplied with $\bm{x}_{0}$. For $u$ chosen uniform random in $[0,1]$, we set $\pi_{d}=\left(1 + 0.1(u - 0.5)\right)\bm{1}^{\top}\bm{y}_{0}$. The activation functions are as in (\ref{softproj}) with $\bm{\beta}=\bm{1}$. The Hopfield sub-iterations were run until error tolerance $10^{-7}$ was achieved with maximum number of sub-iterations being $10^{4}$.   

We fix the problem data generated as above, and perform 100 Monte Carlo simulations, each solving an instance of the dual Hopfield method for the same problem data with randomly chosen $(\lambda_{1}(0),\lambda_{2}(0))$. Fig. \ref{CaseStdyConvergence} shows the convergence trends for the joint vector $\bm{z}:=\left(\bm{x};\bm{y}\right) \in \mathbb{R}^{2n_{\text{G}}}$ associated with the Monte Carlo simulations by plotting the distance between the current iterate $\bm{z}(k)$ and the converged vector $\bm{z}^{*}$ over iteration index $k\in\mathbb{N}$. In particular, Fig. \ref{CaseStdyConvergence} highlights how the geometry induced by the HNN governs the convergence trend: the geodesic distance to minimum $d_{\bm{G}}\left(\bm{z}(k),\bm{z}^{*}\right)$ (\emph{blue} curves in Fig. \ref{CaseStdyConvergence}), where $d_{\bm{G}}$ is given by (\ref{GeodesicDistSoftProj}), decays monotonically, as expected in gradient descent. However, $\|\bm{z}(k)-\bm{z}^{*}\|_{2}$ (\emph{red} curves in Fig. \ref{CaseStdyConvergence}) does not decay monotonically since the Euclidean distance does not encode the Riemannian geometry of the HNN recursion (\ref{AmariNatGradDescent}).

%%%%%%%%%%%%%%%%%%%%%%%%%%%%%%%%%%%%%%%%%%%

\section{Stochastic HNN Flow As Gradient Descent In the Space of Probability Measures}\label{SectionStocHNN}
We next consider the stochastic HNN flow, also known as the ``diffusion machine" \cite[Section 4]{wong1991stochastic}, given by the It\^{o} stochastic differential equation (SDE)
\begin{align}
{\rm{d}}\bm{x} = &\bigg\{-\left(\bm{G}(\bm{x})\right)^{-1}\nabla f(\bm{x}) + T\nabla\left(\left(\bm{G}(\bm{x})\right)^{-1}\bm{1}\right)\bigg\}\: {\rm{d}}t \nonumber\\
&\qquad\qquad\qquad\qquad\quad +\: \sqrt{2T}\left(\bm{G}(\bm{x})\right)^{-1/2}\:{\rm{d}}\bm{w},
\label{HNNsde}	
\end{align}
where the state\footnote{We make the standard assumption that $g_{ii}^{-1}=\sigma_{i}^{\prime}\left(\sigma_{i}(x_{i})\right)=0$ at $x_{i}=0,1$. If the activation functions $\sigma_{i}(\cdot)$ do not satisfy this, then a reflecting boundary is needed at each $x_{i}=0,1$, to keep the sample paths within the unit hypercube; see e.g., \cite{geman1986diffusions}.} $\bm{x}\in(0,1)^{n}$, the standard Wiener process $\bm{w}\in\mathbb{R}^{n}$, the notation $\bm{1}$ stands for the $n\times 1$ vector of ones, and the parameter $T>0$ denotes the thermodynamic/annealing temperature. Recall that the metric tensor $\bm{G}(\bm{x})$ given by (\ref{Mij}) is diagonal, and hence its inverse and square roots are diagonal too with respective diagonal elements being the inverse and square roots of the original diagonal elements.  

The diffusion machine (\ref{HNNsde}) was proposed in \cite{wong1991stochastic,kesidis1995analog} to solve global optimization problems on the unit hypercube, and is a stochastic version of the deterministic HNN flow discussed in Sections \ref{SectionNatGradDescent}--\ref{SectionExample}. This can be seen as a generalization of the Langevin model \cite{gidas1985global,aluffi1985global} for the Boltzmann machine \cite{ackley1985learning}. Alternatively, one can think of the diffusion machine as the continuous-state continuous-time version of the HNN in which Gaussian white noise is injected at each node.

Let $\mathcal{L}_{\rm{FPK}}$ be the Fokker-Planck-Kolmogorov (FPK) operator \cite{risken1989fokker} associated with (\ref{HNNsde}) governing the evolution of the joint PDFs $\rho(\bm{x},t)$, i.e.,
\begin{align}
\dfrac{\partial\rho}{\partial t} = \mathcal{L}_{\rm{FPK}}\rho, \quad \rho(\bm{x},0) = \rho_{0}(\bm{x}),
\label{FPKpde}	
\end{align}
where the given initial joint PDF $\rho_{0}\in\mathcal{P}_{2}\left(\mathcal{M}\right)$. The drift and diffusion coefficients in (\ref{HNNsde}) are tailored so that $\mathcal{L}_{\rm{FPK}}$ admits the stationary joint PDF
\begin{align}
\rho_{\infty}(\bm{x}) = \dfrac{1}{Z(T)}\exp\left(-\dfrac{1}{T}f(\bm{x})\right),
\label{rhoinf}	
\end{align}
i.e., a Gibbs PDF where $Z(T)$ is a normalizing constant (known as the ``partition function") to ensure $\int\rho_{\infty}{\rm{d}}\bm{x}=1$. This follows from noting that for (\ref{HNNsde}), we have\footnote{We remind the readers that the raised indices denote the elements of the inverse of the metric tensor, i.e., $\bm{G}^{-1}=[g^{ij}]_{i,j=1}^{n}$.}
\begin{subequations}
\begin{align}
\mathcal{L}_{\rm{FPK}}\rho &\equiv \displaystyle\sum_{i=1}^{n} \dfrac{\partial}{\partial x_{i}}\bigg\{\rho\left(g^{ii}(x_{i})\dfrac{\partial f}{\partial x_{i}} - T\dfrac{\partial}{\partial x_{i}} g^{ii}(x_{i})\right) \nonumber\\
&\qquad\qquad\qquad\quad + \dfrac{1}{2}\dfrac{\partial}{\partial x_{i}}\left(2Tg^{ii}(x_{i})\rho\right)\bigg\}\label{FPKop1}\\
&=  \displaystyle\sum_{i=1}^{n} \dfrac{\partial}{\partial x_{i}}\bigg\{\rho g^{ii}(x_{i})\dfrac{\partial f}{\partial x_{i}} + Tg^{ii}(x_{i})\dfrac{\partial}{\partial x_{i}}\rho\bigg\},  \label{FPKop2}
\end{align}
\label{FPKoperator}
\end{subequations}
and that $g^{ii}(x_{i})>0$ for each $x_{i}\in(0,1)$. From (\ref{rhoinf}), the critical points of $f$ coincides with that of $\rho_{\infty}$, which is what allows to interpret (\ref{HNNsde}) as an analog machine for globally optimizing the function $f$. In fact, for the FPK operator (\ref{FPKoperator}), the solution $\rho(\bm{x},t)$ of (\ref{FPKpde}) enjoys an exponential rate of convergence \cite[p. 1358-1359]{calogero2012exponential} to (\ref{rhoinf}). We remark here that the idea of using SDEs for solving global optimization problems have also appeared in \cite{brockett1997oscillatory,das1998noisy}.

In the context of diffusion machine, the key observation that part of the drift term can be canceled by part of the diffusion term in the FPK operator (\ref{FPKop1}), thereby guaranteeing that the stationary PDF associated with (\ref{FPKoperator}) is (\ref{rhoinf}), was first pointed out by Wong \cite{wong1991stochastic}. While this renders the \emph{stationary solution} of the FPK operator meaningful for globally minimizing $f$, it is not known if the transient PDFs generated by (\ref{FPKpde}) admit a natural interpretation. We next argue that (\ref{FPKpde}) is indeed a gradient flow in the space of probability measures supported on the manifold $\mathcal{M}$.

\subsection{Wasserstein Gradient Flow}
Using (\ref{FPKop2}), we rewrite the FPK PDE (\ref{FPKpde}) as
\begin{align}
\dfrac{\partial\rho}{\partial t} &= \nabla\cdot\left(\left(\bm{G}(\bm{x})\right)^{-1}\left(\rho\nabla f + T\nabla\rho\right)\right) \nonumber\\
&= \nabla\cdot\left(\rho\left(\bm{G}(\bm{x})\right)^{-1}\nabla\zeta\right),
\label{LaplaceBeltramiLike}	
\end{align}
where 
\begin{align}
\zeta(\bm{x}) := f(\bm{x}) + T\log\rho(\bm{x}).
\label{Definezeta}	
\end{align}
%W.r.t. the volume element $\differential\left(\vol\right) := \sqrt{\det\left(\bm{G}(\bm{x})\right)}\:\differential\bm{x}$ in the Riemannian manifold $\mathcal{M}$, consider the \emph{free energy functional} $F(\rho)$, $\rho\in\mathcal{P}_{2}\left(\mathcal{M}\right)$, given by
For $\rho\in\mathcal{P}_{2}\left(\mathcal{M}\right)$, consider the \emph{free energy functional} $F(\rho)$ given by
\begin{align}
F(\rho) := \mathbb{E}_{\rho}\left[\zeta\right] = \!\displaystyle\int_{\mathcal{M}}\!f\rho\:\differential\bm{x} + T\!\displaystyle\int_{\mathcal{M}}\!\rho\log\rho\:\differential\bm{x},
\label{FreeEnergyDef}	
\end{align}
which is a sum of the \emph{potential energy} $\int_{\mathcal{M}}f\rho\:\differential\bm{x}$, and the \emph{internal energy} $T\int_{\mathcal{M}}\rho\log\rho\:\differential\bm{x}$. The free energy (\ref{FreeEnergyDef}) serves as the Lyapunov functional for (\ref{LaplaceBeltramiLike}) since direct calculation yields (Appendix \ref{subsecdFdt})
\begin{align}
\dfrac{\differential}{\differential t}F(\rho) = - \mathbb{E}_{\rho}\left[\left(\nabla\zeta\right)^{\top}\left(\bm{G}(\bm{x})\right)^{-1}\left(\nabla\zeta\right)\right] \leq 0, 
\label{dFdt}	
\end{align}
along any solution $\rho(\bm{x},t)$ generated by (\ref{LaplaceBeltramiLike}), thanks to $\bm{G}(\bm{x})\succ\bm{0}$ for all $\bm{x}\in\mathcal{M}=(0,1)^{n}$. In particular, the RHS of (\ref{dFdt}) equals zero iff $\nabla\zeta = 0$, i.e., at the stationary PDF (\ref{rhoinf}). For $t<\infty$, the RHS of (\ref{dFdt}) is $<0$ for any transient PDF $\rho(\bm{x},t)$ solving  (\ref{LaplaceBeltramiLike}).

To view (\ref{LaplaceBeltramiLike}) as gradient flow over $\mathcal{P}_{2}(\mathcal{M})$, we will need the notion of (quadratic) \emph{Wasserstein metric} $W_{\bm{G}}$ between two probability measures $\mu$ and $\nu$ on $\mathcal{M}$, defined as
\begin{align}
W_{\bm{G}}\left(\mu,\nu\right) := \left(\underset{\pi\in\Pi_{2}\left(\mu,\nu\right)}{\inf} \displaystyle\int_{\mathcal{M}\times\mathcal{M}} \!\!\left(d_{\bm{G}}\left(\bm{x},\bm{y}\right)\right)^{2}\:\differential\pi(\bm{x},\bm{y})\right)^{\!\!\frac{1}{2}},
\label{DefWass}	
\end{align}
where $\pi(\bm{x},\bm{y})$ denotes a joint probability measure supported on $\mathcal{M}\times\mathcal{M}$, and the geodesic distance $d_{\bm{G}}$ is given by (\ref{GeodesicDist}). The infimum in (\ref{DefWass}) is taken over the set  $\Pi_{2}\left(\mu,\nu\right)$, which we define as the set of all joint probability measures having finite second moment that are supported on $\mathcal{M}\times\mathcal{M}$, with prescribed $\bm{x}$-marginal $\mu$, and prescribed $\bm{y}$-marginal $\nu$. If $\mu$ and $\nu$ have respective PDFs $\rho_{\bm{x}}$ and $\rho_{\bm{y}}$, then we can use the notation $W_{\bm{G}}(\rho_{\bm{x}},\rho_{\bm{y}})$ in lieu of $W_{\bm{G}}(\mu,\nu)$. The subscript in (\ref{DefWass}) is indicative of its dependence on the ground Riemannian metric $\bm{G}(\cdot)$, and generalizes the Euclidean notion of Wasserstein metric \cite{villani2003topics}, i.e., $\bm{G}\equiv\bm{I}$ case. See also \cite[Section 3]{sturm2005convex} and \cite{ohta2009gradient}. Following \cite[Ch. 7]{villani2003topics} and using the positive definiteness of $\bm{G}$, it is easy to verify that (\ref{DefWass}) indeed defines a metric on $\mathcal{P}_{2}\left(\mathcal{M}\right)$, i.e., $W_{\bm{G}}$ is non-negative, zero iff $\mu=\nu$, symmetric in its arguments, and satisfies the triangle inequality.

The square of (\ref{DefWass}) is referred to as the ``optimal transport cost" (see \cite{benamou2000a} for the Euclidean case) that quantifies the minimum amount of work needed to reshape the PDF $\rho_{\bm{x}}$ to $\rho_{\bm{y}}$ (or equivalently, $\rho_{\bm{y}}$ to $\rho_{\bm{x}}$). This can be formalized via the following dynamic variational formula \cite[Corollary 2.5]{lisini2009nonlinear} for (\ref{DefWass}):
\begin{subequations}
\begin{align}
\left(W_{\bm{G}}\left(\mu,\nu\right)\right)^{2} = &\underset{\left(\rho,\bm{u}\right)}{\inf}\:\mathbb{E}_{\rho}\left[\displaystyle\int_{0}^{1}\dfrac{1}{2}\left(\bm{u}(\bm{x},\tau)\right)^{\!\top}\!\!\bm{G}(\bm{x})\bm{u}(\bm{x},\tau)\:\differential\tau\right]\nonumber\\
&\text{subject to}\quad\dfrac{\partial\rho}{\partial\tau} + \nabla\cdot\left(\rho\bm{u}\right) = 0,\label{ConstrLiouville}\\
&\qquad\qquad\quad \rho(\bm{x},\tau=0) = \dfrac{\differential\mu}{\differential\bm{x}},\label{2pbvpFirst}\\
&\qquad\qquad\quad \rho(\bm{x},\tau=1) = \dfrac{\differential\nu}{\differential\bm{x}}.\label{2pbvpSecond}
\end{align}
\label{WGBrenierBenamou}		
\end{subequations}
\noindent The infimum above is taken over the PDF-vector field pairs $(\rho,\bm{u})\in\mathcal{P}_{2}(\mathcal{M})\times\mathcal{U}$, where $\mathcal{U}:=\{\bm{u}:\mathcal{M}\times[0,\infty)\mapsto\mathcal{TM}\}$.
Recognizing that (\ref{ConstrLiouville}) is the continuity/Liouville equation for the single integrator dynamics $\dot{\bm{x}}=\bm{u}$, one can interpret (\ref{WGBrenierBenamou}) as a fixed horizon ``minimum weighted energy" stochastic control problem subject to endpoint PDF constraints (\ref{2pbvpFirst})-(\ref{2pbvpSecond}). 

Following \cite{otto2001the} and \cite[Section 1.2]{lisini2009nonlinear}, we can endow $\mathcal{P}_{2}(\mathcal{M})$ with the metric $W_{\bm{G}}$, resulting in an infinite dimensional Riemannian structure. Specifically, the tangent space $\mathcal{T}_{\rho}\mathcal{P}_{2}(\mathcal{M})$ can be equipped with the inner product $L^{2}_{\bm{G}}$ given by
\begin{align}
\langle \bm{u}, \bm{v} \rangle_{L^{2}_{\bm{G}}} := \displaystyle\int_{\mathcal{M}} \bm{u}^{\top}\bm{G}(\bm{x})\bm{v}\:\rho(\bm{x})\differential\bm{x},
\label{weigtedinnerproduct}	
\end{align}
for tangent vectors $\bm{u}, \bm{v}$. Given a smooth curve $t\mapsto \rho(\bm{x},t)$ in $\mathcal{P}_{2}(\mathcal{M})$, the associated tangent vector $\bm{u}(\bm{x},t)$ is characterized as the vector field that solves (\ref{WGBrenierBenamou}). We refer the readers to \cite[Ch. 13]{villani2008old} for further details. In particular, the ``Wasserstein gradient" \cite[Ch. 8]{ambrosio2008gradient} of a functional $\Phi(\rho)$ at $\rho\in\mathcal{P}_{2}\left(\mathcal{M}\right)$, can be defined as
\begin{align}
\nabla^{W_{\bm{G}}}\Phi(\rho) := -\nabla\cdot\left(\rho\left(\bm{G}(\bm{x})\right)^{-1}\nabla\dfrac{\delta\Phi}{\delta\rho}\right),
\label{WassGrad}	
\end{align}
where $\delta\Phi/\delta\rho$ denotes the functional derivative of $\Phi(\rho)$. Consequently, the FPK PDE (\ref{LaplaceBeltramiLike}) can be transcribed into the familiar gradient flow form:
\begin{align}
\dfrac{\partial\rho}{\partial t} = -\nabla^{W_{\bm{G}}}F(\rho). 
\label{FPKasWassersteinGradFlow}	
\end{align}
Furthermore, the Euler (i.e., first order in time) discretization of the LHS of (\ref{FPKasWassersteinGradFlow}) yields the familiar gradient descent form:
\begin{align}
\rho_{k} = \rho_{k-1} - h\:\nabla^{W_{\bm{G}}}F(\rho)\big\vert_{\rho=\rho_{k-1}},
\label{FPKasWassersteinGradDescent}	
\end{align}
where $\rho_{k} := \rho(\bm{x},t=kh)$ for $k\in\mathbb{N}$. The functional recursion (\ref{FPKasWassersteinGradDescent}) evolves on the metric space $\left(\mathcal{P}_{2}\left(\mathcal{M}\right),W_{\bm{G}}\right)$. We will see next that (\ref{FPKasWassersteinGradDescent}) can be recast as an equivalent proximal recursion, which will prove helpful for computation.

\begin{remark}\label{RemarkNotLaplaceBeltrami}
We clarify here that the FPK operator appearing in the RHS of (\ref{LaplaceBeltramiLike}), or equivalently in the RHS of (\ref{FPKasWassersteinGradFlow}), is not quite same as the Laplace-Beltrami operator \cite{brockett1997notes}
\begin{align}
\dfrac{1}{\sqrt{\det(\bm{G}(\bm{x}))}}\nabla\cdot\left(\sqrt{\det(\bm{G}(\bm{x}))}\:\rho\:\left(\bm{G}(\bm{x})\right)^{-1}\nabla \dfrac{\delta F}{\delta\rho}\right).
\label{LaplaceBeltrami}	
\end{align}	
This is because the It\^{o} SDE (\ref{HNNsde}) was hand-crafted in \cite{wong1991stochastic} such that the associated stationary PDF $\rho_{\infty}$ becomes (\ref{rhoinf}) w.r.t. the volume measure $\differential\bm{x}$, which is desired from an optimization perspective since then, the local minima of $f$ would coincide with the location of the modes of $\rho_{\infty}$. In contrast, (\ref{LaplaceBeltrami}) admits stationary PDF $\exp(-f/T)$ w.r.t. the volume measure $\sqrt{\det(\bm{G}(\bm{x}))}\:\differential\bm{x}$ (up to a normalization constant), and in general, would not correspond to the local minima of $f$.
\end{remark}

\subsection{Proximal Recursion on $\mathcal{P}_{2}\left(\mathcal{M}\right)$}
 A consequence of identifying (\ref{LaplaceBeltramiLike}) as gradient flow in $\mathcal{P}_{2}(\mathcal{M})$ w.r.t. the metric $W_{\bm{G}}$ is that the solution $\rho(\bm{x},t)$ of (\ref{FPKasWassersteinGradFlow}) can be approximated \cite{lisini2009nonlinear} via variational recursion
 \begin{align}
 \varrho_{k} = \underset{\varrho\in\mathcal{P}_{2}\left(\mathcal{M}\right)}{\arg\inf}\:\dfrac{1}{2}W_{\bm{G}}^{2}\left(\varrho_{k-1},\varrho\right) \:+\: h\:F(\varrho), \quad k\in\mathbb{N},
 \label{JKOwithG}	
 \end{align}
with the initial PDF $\varrho_{0}\equiv \rho_{0}(\bm{x})$, and small step-size $h>0$. In particular, $\varrho_{k} \rightarrow \rho(\bm{x},t=kh)$ in strong $L^{1}(\mathcal{M})$ sense as $h\downarrow 0$. That a discrete time-stepping procedure like (\ref{JKOwithG}) can approximate the solution of FPK PDE, was first proved in \cite{jko1998variational} for $\bm{G}\equiv\bm{I}$, and has since become a topic of burgeoning research; see e.g., \cite{ambrosio2008gradient,carlen2003constrained,santambrogio2017gradflows}.

The RHS in (\ref{JKOwithG}) is an infinite dimensional version of the \emph{Moreau-Yosida proximal operator} \cite{moreau1965proximite,rockafellar1976monotone,bauschke2011convex}, denoted as $\prox_{hF}^{W_{\bm{G}}}\left(\cdot\right)$, i.e., (\ref{JKOwithG}) can be written succinctly as
\begin{align}
\varrho_{k} = \prox_{hF}^{W_{\bm{G}}}\left(\varrho_{k-1}\right), \quad k\in\mathbb{N},
\label{defProx}
\end{align}
and (\ref{rhoinf}) is the fixed point of this recursion. Just like the proximal viewpoint of finite dimensional gradient descent, (\ref{JKOwithG}) can be taken as an alternative definition of gradient descent of $F(\cdot)$ on the manifold $\mathcal{P}_{2}\left(\mathcal{M}\right)$ w.r.t. the metric $W_{\bm{G}}$; see e.g., \cite{halder2017cdc,halder2018acc,caluya2019acc,caluya2019tac}. 

The utilitarian value of (\ref{JKOwithG}) over (\ref{FPKasWassersteinGradDescent}) is as follows. While the algebraic recursion (\ref{FPKasWassersteinGradDescent}) involves first order differential operator (w.r.t. $\bm{x}$), the variational recursion (\ref{JKOwithG}) is zero-th order. In the optimization community, the proximal operator is known \cite{parikh2014proximal} to be amenable for large scale computation. In our context too, we will see in Section \ref{SubsecProxAlgo} that (\ref{FPKasWassersteinGradDescent}) allows scattered weighted point cloud-based computation avoiding function approximation or spatial discretization, which will otherwise be impossible had we pursued a direct numerical method for the algebraic recursion (\ref{FPKasWassersteinGradDescent}).

Notice from (\ref{FreeEnergyDef}) that $F(\rho)$ is strictly convex in $\rho\in\mathcal{P}_{2}(\mathcal{M})$. Furthermore, the map $\rho \mapsto W^{2}_{\bm{G}}(\widetilde{\rho},\rho)$ is convex in $\rho$ for any $\widetilde{\rho}\in\mathcal{P}_{2}(\mathcal{M})$. Therefore, (\ref{defProx}) is a strictly convex functional recursion wherein each step admits unique minimizer.

%%%%%%%%%%%%%%%%%%%%%%%%%%%%%%%%%%%%%%%%%%%

\begin{figure*}[t!]
    \centering
    \begin{subfigure}[t]{0.5\textwidth}
        \centering
        \includegraphics[width=\textwidth]{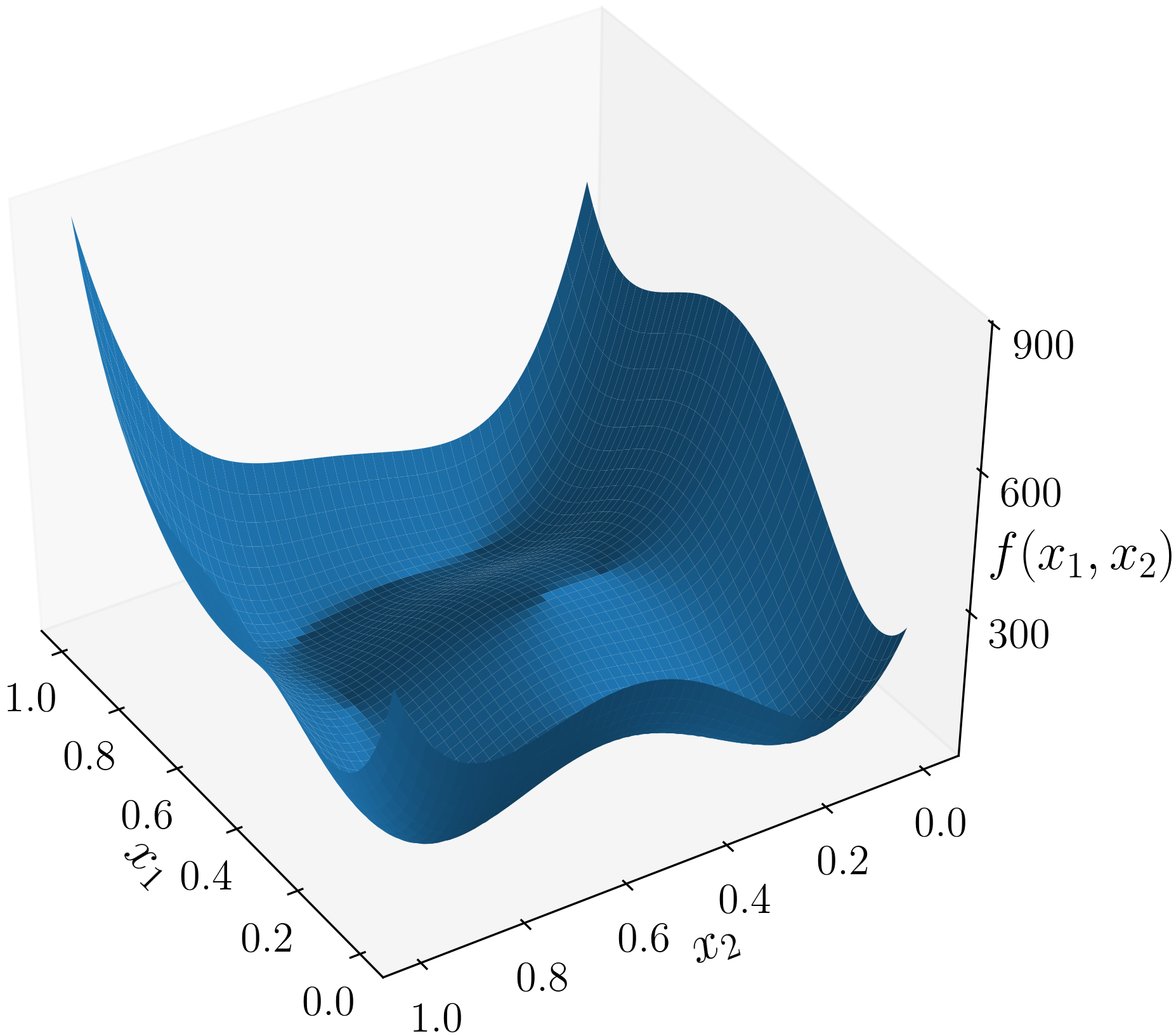}
        \caption{The function $f(x_{1},x_{2})$ given by (\ref{Himmelblau}) over $[0,1]^{2}$.}
    \end{subfigure}%
    ~ 
    \begin{subfigure}[t]{0.5\textwidth}
        \centering
        \includegraphics[width=\textwidth]{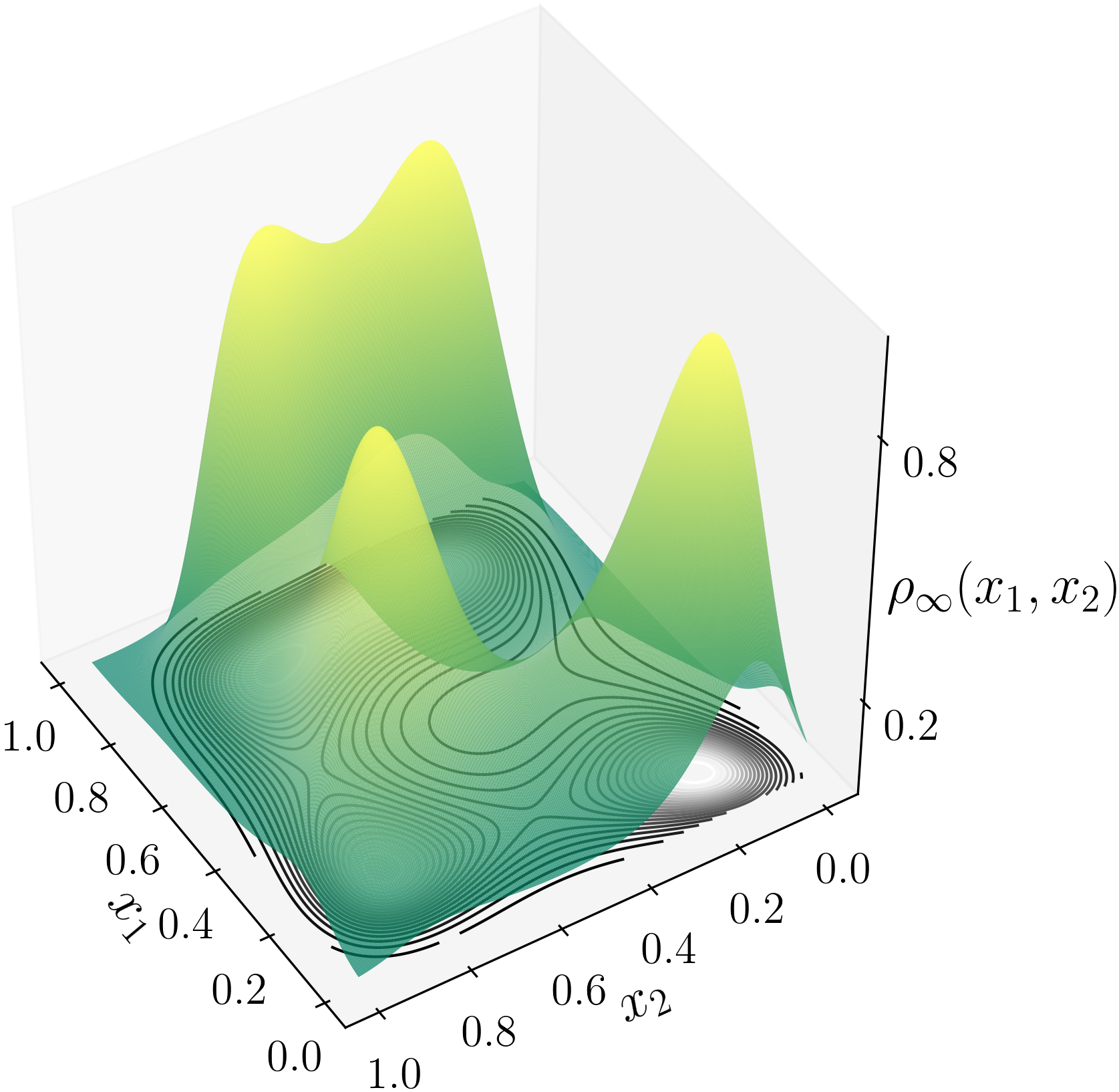}
        \caption{The PDF $\rho_{\infty}(x_{1},x_{2})$ given by (\ref{rhoinf}) with $f$ as in (\ref{Himmelblau}), and $T=25$.}
    \end{subfigure}
\caption{\small{(a) The scaled Himmelblau function $f$ given by (\ref{Himmelblau}) has four local minima in $[0,1]^{2}$. (b) The corresponding stationary PDF $\rho_{\infty}$ has four modes, i.e., local maxima, positioned at the arg\:mins of $f$. The proximal recursion (\ref{JKOwithG}) for the Example in Section \ref{ProxExample} converges to this stationary PDF $\rho_{\infty}$, as shown via the weighted point cloud evolution in Fig. \ref{FigTransientPDFscatterProxExample}.}}
\vspace*{-0.1in}
\label{FigHimmelblauExample}
\end{figure*}

%%%%%%%%%%%%%%%%%%%%%%%%%%%%%%%%%%%%%%%%%%

\subsubsection{Interpretation for small $T$}
In the global optimization context, the parameter $T$ in the free energy (\ref{FreeEnergyDef}) plays an important role. For $T\downarrow 0$, the SDE (\ref{HNNsde}) reduces to the ODE (\ref{AmariNatGradODE}), and (\ref{FreeEnergyDef}) simplifies to $\mathbb{E}_{\rho}[f]$ with $\rho$ being the time-varying Dirac PDF along the solution of (\ref{AmariNatGradODE}). In that case, $W_{\bm{G}}\equiv d_{\bm{G}}$, and hence (\ref{JKOwithG}) reduces to the proximal recursion for \emph{finite dimensional} natural gradient descent (Section \ref{SectionNatGradDescent}):
\begin{align}
\bm{x}_{k} = \underset{\bm{x}\in\mathcal{M}}{\arg\inf}\:\dfrac{1}{2}d_{\bm{G}}^{2}\left(\bm{x}_{k-1},\bm{x}\right) \:+\: h\:f(\bm{x}), \quad k\in\mathbb{N}.
\label{FiniteDimJKO}	
\end{align}
Furthermore, the FPK PDE (\ref{LaplaceBeltramiLike}) reduces to the Liouville PDE 
\[\dfrac{\partial\rho}{\partial t} = \nabla\cdot\left(\left(\bm{G}(\bm{x})\right)^{-1}\rho\nabla f\right),\]
whose stationary PDF becomes an weighted sum of Diracs located at the stationary points of $f$. The stationary solution of (\ref{FiniteDimJKO}) converges to this stationary PDF. Thus, (\ref{JKOwithG}) with small $T$ approximates the deterministic natural gradient descent.

\subsubsection{Interpretation for large $T$}
For $T$ large, the second summand in (\ref{FreeEnergyDef}) dominates, i.e., (\ref{JKOwithG}) reduces to maximum entropy descent in $\mathcal{P}_{2}(\mathcal{M})$ w.r.t. the metric $W_{\bm{G}}$. The resulting proximal recursion admits uniform PDF as the fixed point, which is indeed the $T\rightarrow\infty$ limit of the stationary PDF (\ref{rhoinf}). Thus, (\ref{JKOwithG}) with large $T$ approximates pure diffusion.

%%%%%%%%%%%%%%%%%%%%%%%%%%%%%%%%%%%%%%%%%%%

\begin{figure*}[htpb]
\centering
\includegraphics[width=\linewidth]{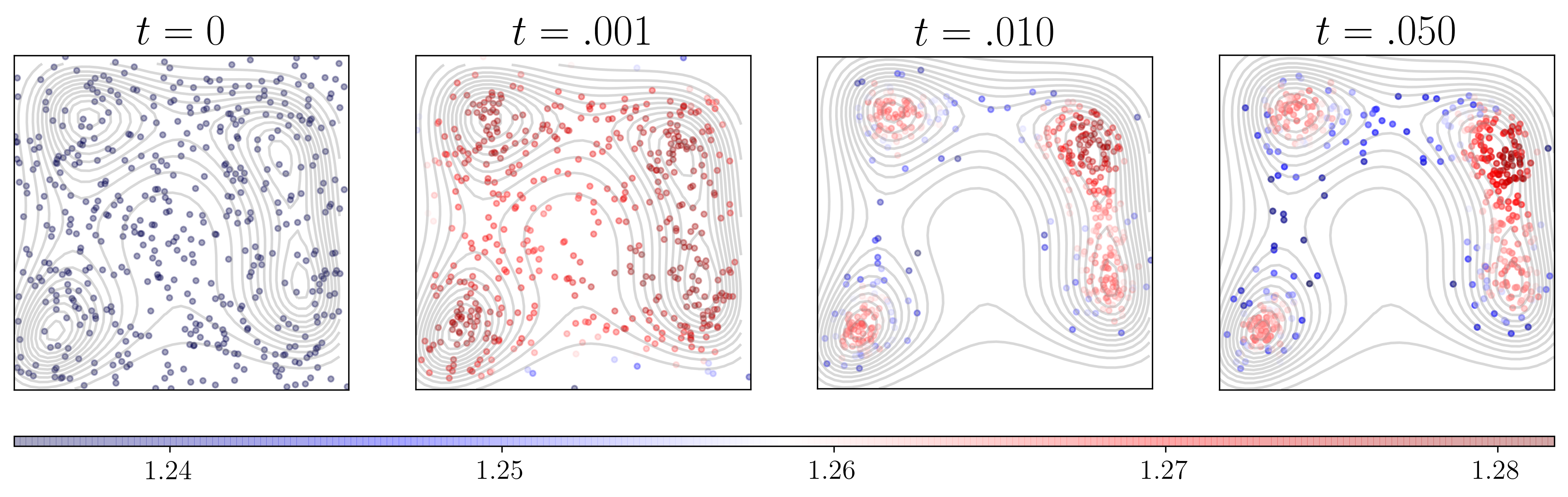}
\vspace*{-0.1in}
\caption{\small{The proximal (weighted scattered point cloud) joint PDF evolution over $[0,1]^{2}$ associated with (\ref{tanhSDE}) and (\ref{Himmelblau}), starting from the uniform initial joint PDF. The joint PDF evolution is computed via (\ref{JKOwithG}) with $W_{\bm{G}}$ given by (\ref{DefWass}), wherein $d_{\bm{G}}$ is given by (\ref{GeodesicDistSoftProj}). The color (\emph{red = high, blue = low}) denotes the joint PDF value at a point at that time (\emph{see colorbar}). The location of a point in $[0,1]^{2}$ is governed by (\ref{tanhSDE}). In the sub-figures above, the background contour lines correspond to the stationary PDF (\ref{rhoinf}). The simulation details are given in Section \ref{ProxExample}.}}
\vspace*{-0.15in}
\label{FigTransientPDFscatterProxExample}
\end{figure*}

%%%%%%%%%%%%%%%%%%%%%%%%%%%%%%%%%%%%%%%%%%

\subsection{Proximal Algorithm}\label{SubsecProxAlgo}
The recursion (\ref{JKOwithG}) is not only appealing from theoretical perspective in that it reveals the metric geometry of gradient flow, but also from an algorithmic perspective in that it opens up the possibility to compute the transient solutions of the PDE (\ref{LaplaceBeltramiLike}) in a scalable manner via convex optimization. Similar to the finite dimensional proximal algorithms \cite{parikh2014proximal}, the recently introduced infinite dimensional proximal algorithms \cite{caluya2019acc,caluya2019tac} were shown to be amenable for large scale implementation. Next, we adapt the proximal algorithm in \cite[Section III.B]{caluya2019tac} to solve (\ref{JKOwithG}) for an example problem.

\subsubsection{Example for stochastic HNN}\label{ProxExample}
We consider the activation functions $\sigma_{i}(\cdot)$ as in Section \ref{SectionExample}, resulting in 
\begin{align}
g^{ii} = g^{-1}_{ii} \overset{(\ref{HessianExample})}{=} 2\beta_{i}x_{i}(1-x_{i}), \: \beta_{i}>0, \quad i=1,\hdots,n,
\label{Examplegiiinv}	
\end{align}
which in turn gives the following instance of (\ref{HNNsde}):
\begin{align}
{\rm{d}}x_{i} = &\bigg\{-2\beta_{i}x_{i}\left(1-x_{i}\right)\nabla f(\bm{x}) + 2\beta_{i}T\left(1-2x_{i}\right)\bigg\}\: {\rm{d}}t \nonumber\\
&\qquad +\: \sqrt{4T\beta_{i}x_{i}\left(1-x_{i}\right)}\:{\rm{d}}w_{i}, \quad i=1,\hdots,n.
\label{tanhSDE}	
\end{align}
To compute the transient joint PDF flow associated with (\ref{tanhSDE}), we solve (\ref{JKOwithG}) with $W_{\bm{G}}$ given by (\ref{DefWass}), wherein $d_{\bm{G}}$ is given by (\ref{GeodesicDistSoftProj}).

For numerical simulation, we set $n=2$, $T=25$, $\beta_{i}=1/4$ for $i=1,2$, and
\begin{align}
f\left(x_{1},x_{2}\right) =& \left(\left(10 x_{1} - 5\right)^{2} + 10 x_{2} - 16\right)^{2} \nonumber\\
& \qquad+ \left(10 x_{1} - 12 + \left(10 x_{2} - 5\right)^{2}\right)^{2},
\label{Himmelblau}	
\end{align}
where $(x_{1},x_{2})\in[0,1]^{2}$. The function $f$ in (\ref{Himmelblau}) is a re-scaled version of the so-called \emph{Himmelblau's function} \cite{himmelblau1972applied}, often used as a benchmark for non-convex optimization. As seen in Fig. \ref{FigHimmelblauExample}, the function $f$ has four local minima, and the corresponding stationary PDF (\ref{rhoinf}) has four modes (i.e., local maxima) at the arg\:mins of $f$.

We generate $N=500$ samples from the uniform initial joint PDF $\rho_{0}$ supported on $[0,1]^{2}$ in the form of the point cloud $\{\bm{x}_{0}^{i},\varrho_{0}^{i}\}_{i=1}^{N}$.  Here, $\{\bm{x}_{0}^{i}\}_{i=1}^{N}$ denotes the location of the samples in $[0,1]^{2}$ while $\{\varrho_{0}^{i}\}_{i=1}^{N}$ denotes the values of the initial joint PDF evaluated at these locations. We apply the gradient flow algorithm in \cite[Section III.B]{caluya2019tac} to recursively compute the updated point clouds $\{\bm{x}_{k}^{i},\varrho_{k}^{i}\}_{i=1}^{N}$ for $k\in\mathbb{N}$. Specifically, we update the sample locations $\{\bm{x}_{k}^{i}\}_{i=1}^{N}$ via the Euler-Maruyama\footnote{Instead of the Euler-Maruyama, it is possible to use any other stochastic integrator for this part of computation; see \cite[Remark 1 in Section III.B]{caluya2019tac}.} scheme \cite[Ch. 10]{kloeden2013} applied to (\ref{tanhSDE}) with time-step $h=10^{-4}$, i.e., for $k\in\mathbb{N}$,
\begin{align}
&\bm{x}_{k}^{i} = \:\bm{x}_{k-1}^{i} + 2h\bm{\beta}\odot\bigg\{\!\!- \bm{x}_{k-1}^{i}\odot\left(\bm{1}-\bm{x}_{k-1}^{i}\right)\odot\nabla f(\bm{x}_{k-1}^{i}) \nonumber\\
&+ T\left(\bm{1}-2\bm{x}_{k-1}^{i}\right)\!\!\bigg\} + \sqrt{4T\bm{\beta}\odot\bm{x}_{k}^{i}\odot(\bm{1}-\bm{x}_{k}^{i})}\odot\Delta\bm{w}_{k}^{i},
\label{EulerMaruyamaStochasticHNN}	
\end{align}
where $\{\Delta\bm{w}_{k}^{i}\}_{i=1}^{N}$ are independent and identically distributed samples from a Gaussian joint PDF with zero mean and covariance $h\bm{I}$. To compute the joint PDF values $\{\varrho_{k}^{i}\}_{i=1}^{N}$ at the updated sample locations, we apply the gradient flow algorithm in \cite[Section III.B]{caluya2019tac} to solve the proximal recursion (\ref{JKOwithG}) with entropic-regularization parameter $\epsilon = 0.1$. This results in a scattered weighted point cloud evolution as depicted in Fig. \ref{FigTransientPDFscatterProxExample}, and thanks to the exponential rate-of-convergence \cite{calogero2012exponential} to the stationary PDF, the local minima of $f$ can be quickly ascertained from the evolution of the modes of the transient joint PDFs. Fig. \ref{FigCompTimeProxExample} shows that the resulting computation has extremely fast runtime due to certain nonlinear contraction mapping property established in \cite{caluya2019tac}.

An appealing feature of this proximal gradient descent computation is that it does not involve spatial discretization or function approximation, and is therefore scalable to high dimensions. In particular, we discretize time but not the state space, and evolve weighted point clouds over time. We eschew the algorithmic details and refer the readers to \cite{caluya2019tac}.

This example demonstrates that the HNN SDE (\ref{HNNsde}) and its associated proximal recursion (\ref{JKOwithG}) for gradient descent in $\mathcal{P}_{2}(\mathcal{M})$ can be a practical computational tool for global optimization.

%%%%%%%%%%%%%%%%%%%%%%%%%%%%%%%%%%%%%%%%%%%

\begin{figure}[htpb]
\centering
\includegraphics[width=\linewidth]{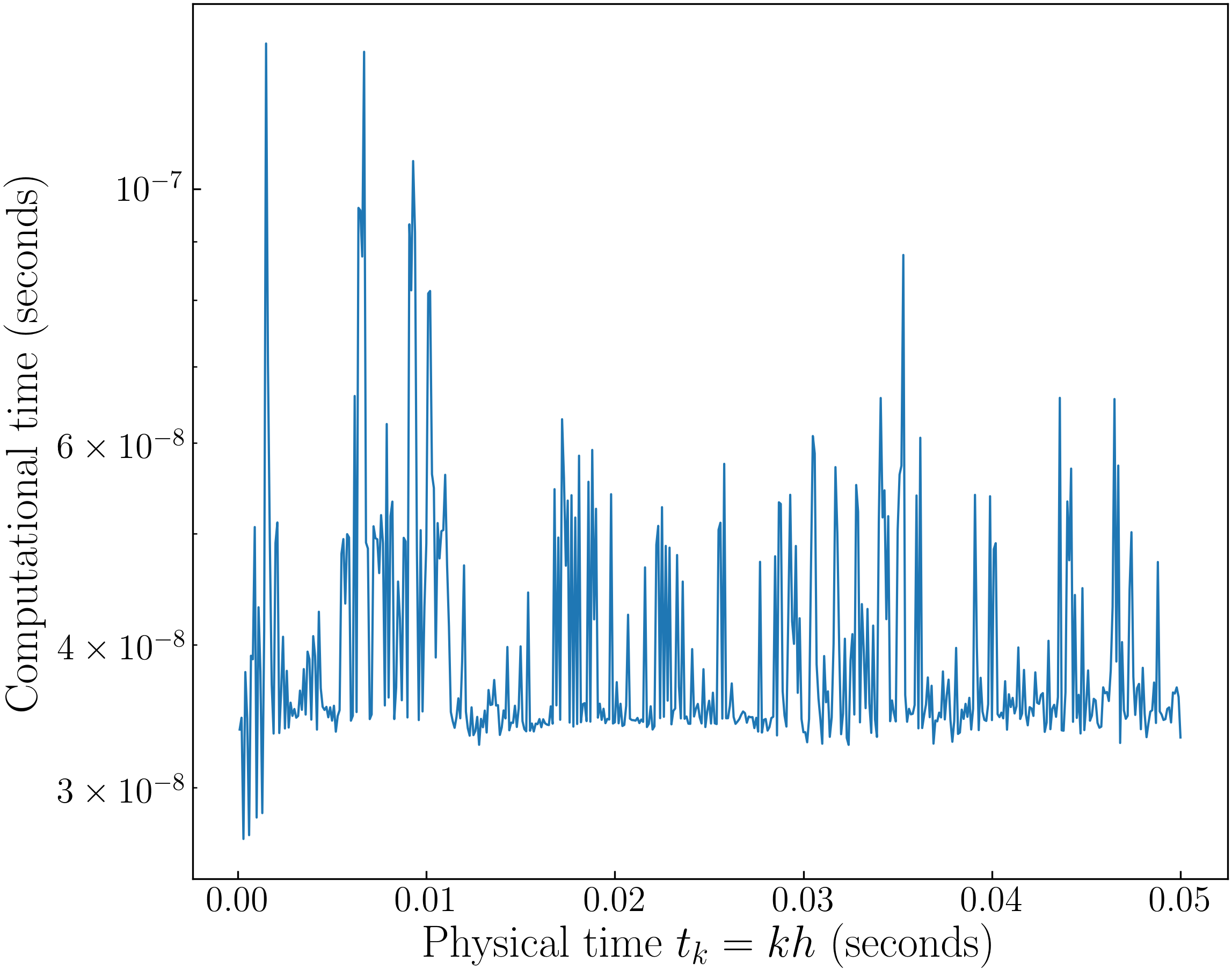}
\vspace*{-0.1in}
\caption{\small{The computational runtime associated with the proximal recursion (\ref{JKOwithG}) for the numerical example given in Section \ref{ProxExample}. Here, the physical time-step $h=10^{-4}$, and $k\in\mathbb{N}$.}}
\vspace*{-0.1in}
\label{FigCompTimeProxExample}
\end{figure}

%%%%%%%%%%%%%%%%%%%%%%%%%%%%%%%%%%%%%%%%%%

\setlength{\tabcolsep}{12pt}
\renewcommand{\arraystretch}{1.8}
\begin{table*}%
\centering
\begin{tabular}{| c | c | c |}
\hline\hline
Attributes &   Gradient Flow in Deterministic HNN &   Gradient Flow in Stochastic HNN  \\[5pt]
\hline
\hline
Graphical illustration     
    &
\raisebox{-\totalheight}{\includegraphics[width=0.3\textwidth]{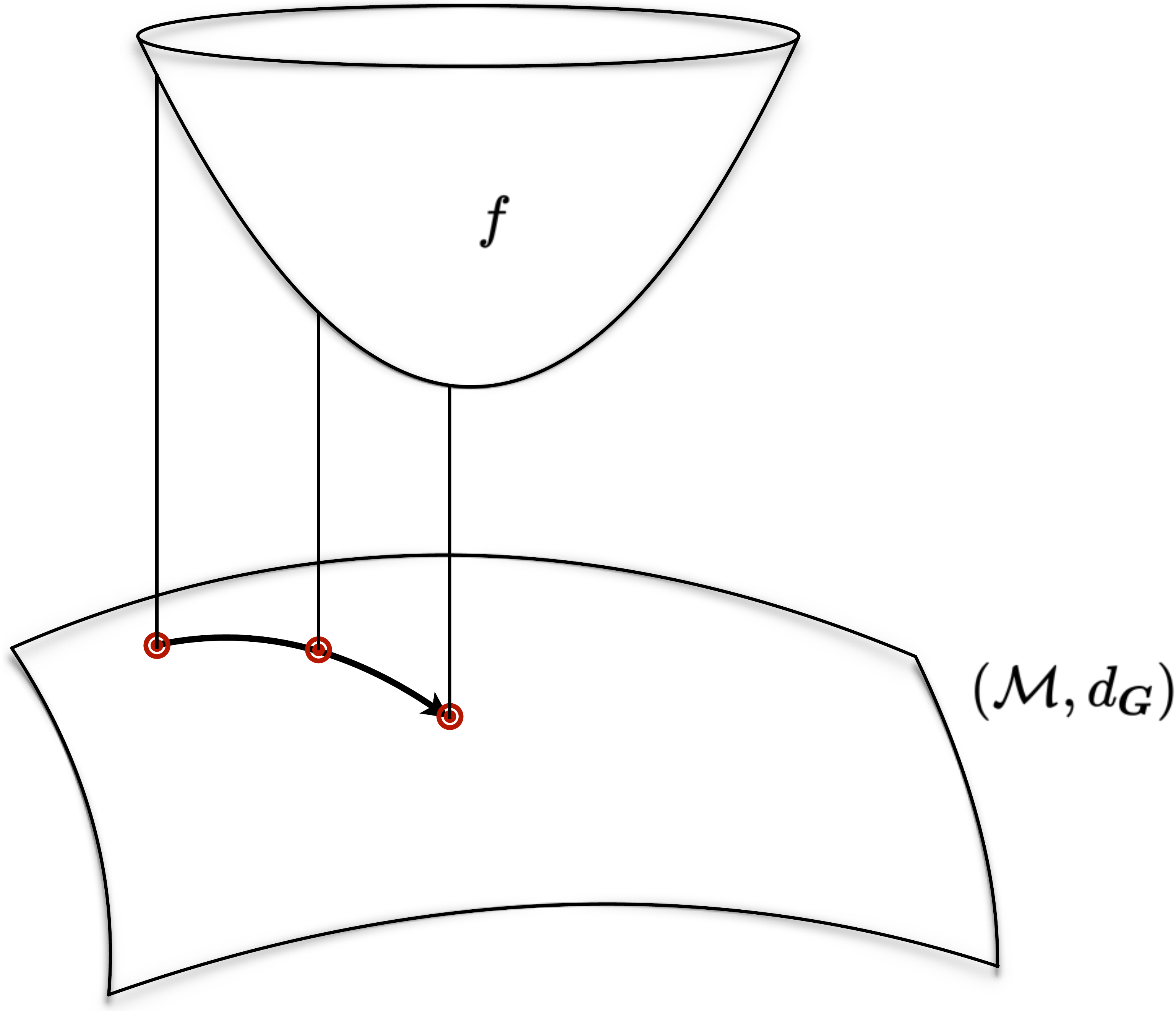}}
    &
\raisebox{-\totalheight}{\includegraphics[width=0.3\textwidth]{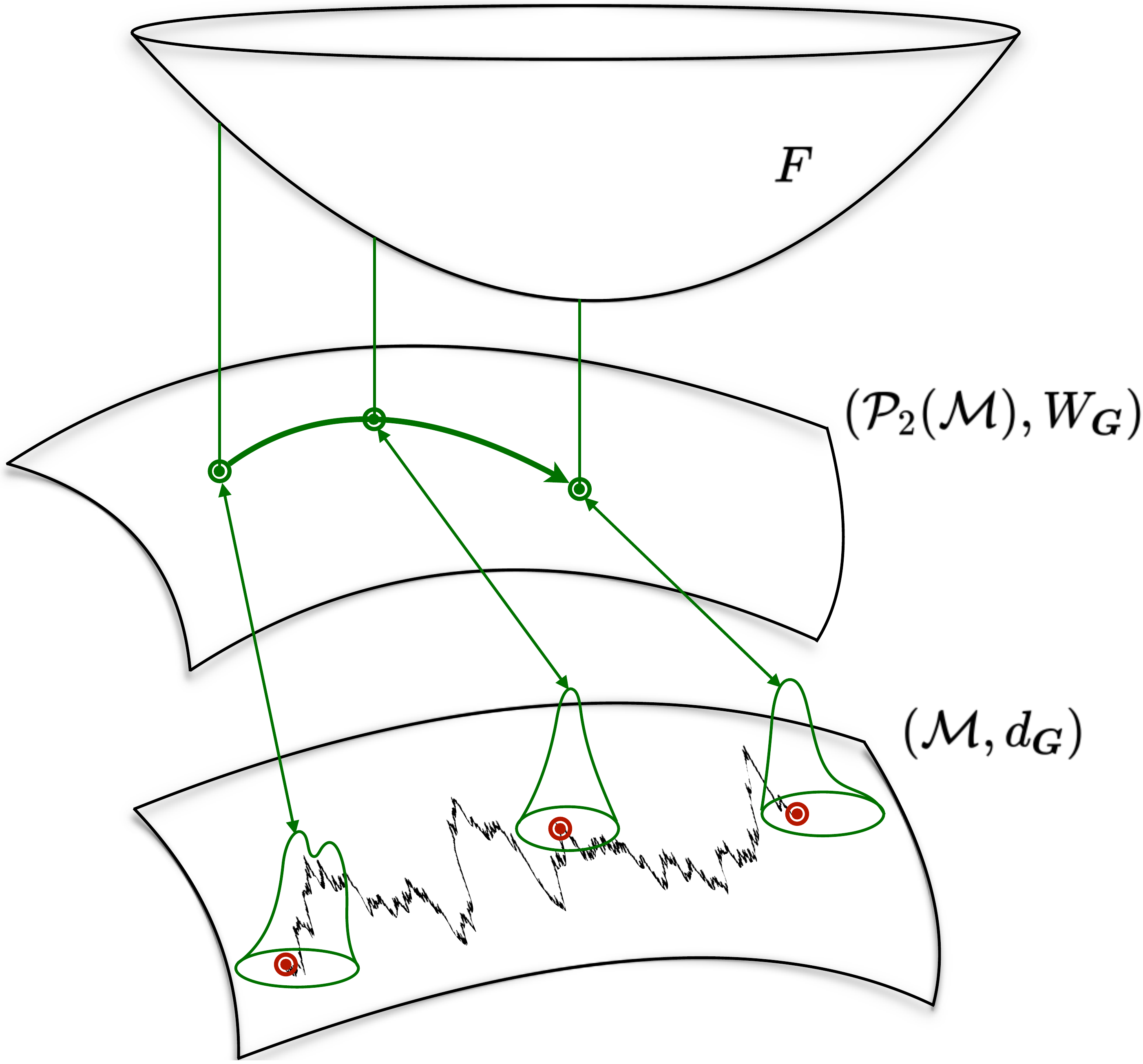}}   \\[5pt]
\hline
Variational problem & $\underset{\bm{x}\in\mathcal{M}}{\inf}\:f(\bm{x})$ & $\underset{\rho\in\mathcal{P}_{2}\left(\mathcal{M}\right)}{\inf}\:F(\rho)$\\[5pt]
\hline
Objective & $f(\bm{x})$ & $F(\rho)$ given by (\ref{FreeEnergyDef})\\[5pt]
\hline
Distance & $d_{\bm{G}}$ given by (\ref{GeodesicDist}) & $W_{\bm{G}}$ given by (\ref{DefWass})\\[5pt]
\hline
Differential equation & $\dfrac{\differential\bm{x}}{\differential t} = -\left(\bm{G}(\bm{x})\right)^{-1}\nabla f$ & $\dfrac{\partial\rho}{\partial t} = \nabla\cdot\left(\rho\left(\bm{G}(\bm{x})\right)^{-1}\dfrac{\delta F}{\delta\rho}\right)$ \\
& $:=-\nabla^{d_{\bm{G}}}f(\bm{x})$& $\hspace*{-0.3in}:=-\nabla^{W_{\bm{G}}}F(\rho)$\\[5pt]
\hline
Algebraic recursion & $\bm{x}_{k} = \bm{x}_{k-1} - h\nabla^{d_{\bm{G}}}f(\bm{x})\big\vert_{\bm{x}=\bm{x}_{k-1}}$ & $\rho_{k} = \rho_{k-1} - h\nabla^{W_{\bm{G}}}F(\rho)\big\vert_{\rho=\rho_{k-1}}$\\[5pt]
\hline
Proximal recursion & $\bm{x}_{k} = \underset{\bm{x}\in\mathcal{M}}{\arg\inf}\:d_{\bm{G}}^{2}\left(\bm{x},\bm{x}_{k-1}\right) + h f(\bm{x})$ & $\rho_{k} = \underset{\rho\in\mathcal{P}_{2}\left(\mathcal{M}\right)}{\arg\inf}\:W_{\bm{G}}^{2}\left(\rho,\rho_{k-1}\right) + h F(\rho)$\\
& $\hspace{-0.55in}:=\prox_{hf}^{d_{\bm{G}}}\left(\bm{x}_{k-1}\right)$ & $\hspace*{-0.7in}:=\prox_{hF}^{W_{\bm{G}}}\left(\rho_{k-1}\right)$\\[5pt]
\hline
\end{tabular}
\caption{Comparison between the natural (finite dimensional) gradient descent for the deterministic HNN and the measure-valued (infinite dimensional) gradient descent for the stochastic HNN. The graphical illustrations in the first row show that the finite dimensional gradient flow for the deterministic HNN evolves on $\mathcal{M}$, and is a gradient descent of $f$ w.r.t. the distance $d_{\bm{G}}$. On the other hand, the infinite dimensional gradient flow for the stochastic HNN evolves on $\mathcal{P}_{2}(\mathcal{M})$, and is a gradient descent of $F$ w.r.t. the distance $W_{\bm{G}}$. Notice that for stochastic HNN, the sample paths of the SDE (\ref{HNNsde}) in $\mathcal{M}$, induce a flow of PDFs in $\mathcal{P}_{2}(\mathcal{M})$.}
\label{TableGradFlowHNN}
\vspace*{-0.15in}
\end{table*}

%%%%%%%%%%%%%%%%%%%%%%%%%%%%%%%%%%%%%%%%%%%

\section{Conclusions}
In this paper, we highlight the metric geometry of the dynamics associated with the deterministic and the stochastic Hopfield Neural Networks (HNN). We show that the deterministic HNN flow is a natural gradient descent with respect to a metric tensor that depends on the choice of activation functions. Alternatively, the same deterministic HNN flow can be interpreted as mirror descent on a suitable manifold that again is governed by the activation functions. The stochastic HNN, also known as the ``diffusion machine", defines a Wasserstein gradient flow on the space of probability measures induced by the underlying stochastic differential equation on the ground Riemannian manifold. This particular viewpoint leads to infinite dimensional proximal recursion on the space of probability density functions that provably approximates the same Wasserstein gradient flow in the small time step limit. We provide a numerical example to illustrate how this proximal recursion can be implemented over probability weighted scattered point cloud, obviating spatial discretization or function approximation, and hence is scalable. Our numerical experiments reveal that the resulting implementation is computationally fast, and can serve as a practical stochastic optimization algorithm. A conceptual summary of the gradient descent interpretations for the deterministic and the stochastic HNN is outlined in Table \ref{TableGradFlowHNN}.

We clarify here that our intent in this paper has been understanding the basic geometry underlying the differential equations for the HNN flow. We hope that the results of this paper will be insightful to the practitioners in the optimization and machine learning community, and motivate further theoretical and algorithmic development of global optimization via neural networks.

%%%%%%%%%%%%%%%%%%%%%%%%%%%%%%%%%%%%%%%%%%%

\appendix

\subsection{Derivation of (\ref{dFdt})}\label{subsecdFdt}
We start by noting that
\begin{align}
\dfrac{\differential F}{\differential t} = \!\!\displaystyle\int_{\mathcal{M}}\!\dfrac{\delta F}{\delta\rho}\dfrac{\partial\rho}{\partial t}\:\differential\bm{x} &\overset{(\ref{FPKasWassersteinGradFlow})}{=}\!\!\displaystyle\int_{\mathcal{M}}\!\dfrac{\delta F}{\delta\rho}\cdot\left(-\nabla^{W}_{\bm{G}}F(\rho)\right)\differential\bm{x},\nonumber\\
&\overset{(\ref{WassGrad})}{=}\!\!\displaystyle\int_{\mathcal{M}}\!\dfrac{\delta F}{\delta\rho}\nabla\cdot\left(\rho{\bm{G}}^{-1}\nabla\dfrac{\delta F}{\delta\rho}\right)\differential\bm{x}, \nonumber\\
&= -\!\!\displaystyle\int_{\mathcal{M}}\!\!\left(\nabla\dfrac{\delta F}{\delta\rho}\right)^{\!\!\top}\!\!\!\rho\bm{G}^{-1}\nabla\dfrac{\delta F}{\delta\rho}\:\differential\bm{x},\label{dFdtintermed}	
\end{align}
where the last step follows from integration-by-parts w.r.t. $\bm{x}$, and zero flux of probability mass across the boundary of $\mathcal{M}$.

For $F(\rho)$ given by (\ref{FreeEnergyDef}), $\delta F/\delta\rho = f + T\left(1+\log\rho\right)$, and therefore,
\begin{align}
\nabla\dfrac{\delta F}{\delta\rho} = \nabla\left(f + T\log\rho\right)  \overset{(\ref{Definezeta})}{=} \nabla\zeta.
\label{nablaFderivative}	
\end{align}
Combining (\ref{dFdtintermed}) and (\ref{nablaFderivative}), we arrive at (\ref{dFdt}).

\end{document}